\documentclass[16pt,reqno]{article} 
\usepackage{mathrsfs,times}

\date{}
\usepackage[english]{babel}
\usepackage[ansinew]{inputenc}
\usepackage{amsmath,amsthm,amsfonts,amssymb,graphicx}
\usepackage{graphics}
\usepackage{color}
\usepackage{algorithm,algorithmic}
\usepackage{graphicx}
\usepackage{epstopdf}
\usepackage{caption}
\usepackage{cases}
\usepackage{hyperref}

\usepackage{amssymb,amsthm,amsmath}
\usepackage[numbers,sort&compress]{natbib}

 \newtheorem{theorem}{Theorem}[section]
 
 \newtheorem{lemma}[theorem]{Lemma}
 
 \theoremstyle{definition}
 \newtheorem{definition}[theorem]{Definition}
 \theoremstyle{remark}

 \theoremstyle{eg}
 \newtheorem{example}[theorem]{Example}
 
 \theoremstyle{fact}
 
\numberwithin{equation}{section}

\title{\Large\bf  Approximation analysis  for the minimization problem of difference-of-convex functions with Moreau envelopes
\vskip2mm
\small {Dedicated to Professor Rockafellar R.T. for his 90th birthday}}
\author{Yan Tang$^{1,2}$, Shiqing Zhang$^{1}$\\
$^{1}$ College of Mathematics, Sichuan University, Chengdu 610065, China\\(zhangshiqing@scu.edu.cn)\\
 $^{2}$ School of Mathematics and Statistics,
Chongqing Technology \\
and Business University,Chongqing 400067, China. (ttyy7999@163.com)\\
}

\begin{document}
\maketitle

\begin{abstract}
In this work  the minimization problem  for the difference of convex (DC) functions is studied by using  Moreau envelopes and  the  descent method with Moreau gradient is employed  to approximate the numerical  solution.  The main regularization idea in this work is inspired by   Hiriart-Urruty \cite{HU1991}, Moudafi\cite{M2023}, regularize the components of the DC problem by adapting  the  different parameters and strategic matrices flexibly to evaluate    the  whole DC problem.  It is shown that the   inertial gradient  method  as well as the classic gradient descent scheme tend  towards  an  approximation stationary point of the original problem. 



\end{abstract}


\section{Introduction}
In this paper, we are concerned with the Difference-of-convex {\it DC}  optimization problem, which 
 reads as  
\begin{eqnarray}\label{Eq:inclusion}
\inf_{x\in \mathbb{R}^n}\Phi(x)={g(x)-f(x)}.
\end{eqnarray}
where $f, g$ are  two convex  functions
\vskip1mm
Up to now, the difference-of-convex (DC) optimization problem has received widespread attention due to its various applications, such as digital communication systems (Alvarado et al. \cite{ASP2014}), allocation and power allocation \cite{PE2010}, and compressed sensing (Yin et al.\cite{YLHX2015}, Beck and Teboulle\cite{BT2009},  Bertsekas\cite{B1999}), multi-channel networks, image restoration processing, discrete tomography, and clustering, and it seems particularly suitable for modeling some nonconvex industrial problems.
\vskip1mm
  Solving the DC (difference of convex functions) program in the past decades mainly relies on the combination method for solving global continuous optimization ( involving finding global solutions for nonconvex models), and   the convex analysis method for solving noconvex programming which mainly originated from Pham Dinh Tao's work in 1974 on the calculation of the bounded norm of matrices (i.e. maximizing the semi norm on the unit sphere of the norm) and after that some mathematicians  extensively studied and introduced subgradient algorithms for solving convex maximization  problems in nonsmooth and nonconvex optimization(\cite{T1981, T1984, T1985, T1986} and their references). These works are extended to DC (difference of convex functions) programs in a natural and reasonable way. 
  \vskip1mm
Compared with the combination methods that have studied many global algorithms, there are few algorithms in convex analysis methods for solving DC programs.  Several recent works seem to have proposed novel inexact approaches suited for convex minimization problems. In An and Tao \cite{AT2005}, they considered the  primal and dual problem and constructed the  primal and dual solutions  sequences $\{x_n\}$ and $\{y_n\}$
\begin{eqnarray*}\label{Eq:inclusion2}
y_n\in argmin_{y}{f^*(y)-g^*(y_{n-1})-\langle x_n,y-y_{n-1}\rangle},\\
x_{n+1}\in argmin_{x}{g(x)-f(x_n)-\langle x-x_n,y_n\rangle},
\end{eqnarray*}
where  $g^*, h^*$ denote
the conjugate functions of $g$ and $h$, respectively. The relevant achievements of the dual method can also be found in Rockafellar\cite{R1966}.
\vskip1mm

Regularization techniques in DC programming which have been first studied by  Tao in 1986\cite{T1988}   while studying numerical algorithms and they were used to improve the DC algorithm in solutions of many real world nonconvex programs,  see Dinh and An\cite{DA1997},  Hiriart-Urruty\cite{HU1985}, Tao \cite{T1984,T1985} and references therein.
\vskip1mm
\vskip1mm

In these celebrated results, Moreau regularization has played a major role. In recent years, the development of statistical machine learning algorithms and other applications has also shown that Moreau regularization have been
prevalent  due to their nice properties, better generalization ability and concise calculation of gradient.
  The Moreau regularization  for a convex function $g$ is 
\begin{eqnarray*}
g^\lambda(x)=\inf_{z\in \mathbb{R}^n} g(z)+\frac{1}{2\lambda}\|z-x\|^2.
\end{eqnarray*}
\vskip1mm
Although nonsmooth weakly lower semicontinuous convex functions can be smoothed through their Moreau envelope, applying that directly to the DC problem $\Phi=g-f$ as a whole may be unreliable. On the one hand, the proximal mapping of $\Phi$ may be difficult to calculate and even undefined; on the other hand, due to the concave component $-f $, the Moreau envelope of $\Phi$ may not be smooth. 
\vskip1mm
 To conqure this drawback, under the motivation that smoothing each component of $\Phi$ separately will surely give a smooth DC function, Sun and Sun\cite{SS2021} studied a smoothing approximation of a general DC function called Moreau envelope difference (DME) smoothing,  where both components $g$ and $f$ of the DC function are replaced by their respective Moreau envelopes:
\begin{eqnarray}\label{Eq:inclusion3}
\inf_{x\in \mathbb{R}^n}\Phi_{\lambda}(x)=g^{\lambda}(x)-f^{\lambda}(x),
\end{eqnarray}
where $g=g_1+g_2$ with $g_1$ continuous, Lipshictz differential and $g_2$ proper closed convex,and
\begin{eqnarray*}
&&g^{\lambda}(x)=\min_{z\in \mathbb{R}^n}\{\langle\nabla g_1(x),z\rangle+g_2(z)+\frac{1}{2\lambda}\|z-x\|^2\},\\
&&f^{\lambda}(x)=\min_{z\in\mathbb{R}^n}\{f(z)+\frac{1}{2\lambda}\|z-x\|^2\}. 
\end{eqnarray*} 
Subsequently, Sun and Sun's method was extended by Moudafi\cite{M2023} and a beautiful achievement is presented, specifically, the selection of the different  parameters,  maximizing the flexibility of regularization to approximate the solution of  problem (\ref{Eq:inclusion}). For the nonconvex optimization, the readers can refer the excellent works of Bonettini et al.\cite{BPR2020}, Chen \cite{C2013}, Toland \cite{T1978,T1979}.
\vskip1mm

In view of the success of Moreau regularization techniques in approximating the nonconvex optimization problems, and  the moderate smoothness of the penalty term, in this paper, we incorporate regular terms with optional strategy and different parameters  to generalize the Moreau envelop, that is,
\begin{eqnarray}\label{Eq:inclusion4}
\inf_{x\in \mathbb{R}^n}\Phi_{\lambda,\mu}(x)=g_{\lambda,D_1}(x)-f_{\mu,D_2}(x),
\end{eqnarray}
where $g_{\lambda,D_1}(x)$ and $f_{\mu,D_2}(x)$ standing for the Moreau envelopes of $g$, $f$  induced by $\lambda,\mu$ and strategic matrices $D_i(i=1,2)$, respectively.  And, inspired by recently works on the  so-called heavy ball method  which  has been introduced  and been translated,  modified and generalized successively by Polyak \cite{P1964},  Alvarez\cite{A2004}, Alvarez and Attouch \cite{AA2001},  Nesterov\cite{N1983},  Moudafi and Oliny \cite{MO2003}, Guler\cite{G1991},  Beck and Teboulle \cite{BT2009},  two  parallel proximal algorithms are proposed and the approximation analysis are obtained.
\vskip 1mm
The outline of the paper is as follows. In Section \ref{Sec:Pre}, we collect some definitions and results needed for our analysis. In Section \ref{Sec:Alg}, the properties of  Moreau envelopes  induced by $\lambda,\mu$ and $D_i(i=1,2)$ are studied,  two parallel algorithms based on the classical gradient descent method are also proposed and the approximation analysis are obtained. Finally, in Section 4 numerical example illustrates the performances of our scheme.
\vskip 1mm
\section{Some definitions and lemmas}\label{Sec:Pre}
Let $X$ be  the n-dimensional Euclidean space $\mathbb{R}^n$ with inner product $\langle \cdot,\cdot\rangle$ and
Euclidean norm $\|\cdot\|$.  Let $\mathcal{S_{++}}$ be the set of symmetric positive definite matrices. For $m\geq 1$, let $M_m\subset \mathcal{S_{++}}$ be the set of all symmetric positive definite matrices with  eigenvalues contained in  $[\frac{1}{m},m]$. Agree that the norm of $D$ is the largest eigenvalue $m$.  For any $D\in M_m$, we have $D^{-1}\in M_m $, and 
\begin{eqnarray}\label{Eq2.1}
\frac{1}{m}\|x\|^2\leq\|x\|^2_{D}\leq m\|x\|^2,
\end{eqnarray}
 where the norm in the metric induced by $D$ is $\|x\|_{D}=\sqrt{x^TDx}$.

\begin{definition}\label{Def:2.1}
(Clarke\cite{C2013}) A function $g:X\rightarrow \mathbb{R}$ is said  to be lower semi-continuous if
\begin{eqnarray*}
g(u)\leq\lim_{n\rightarrow \infty}\inf g(u_n),
\end{eqnarray*}
for each $u\in X$.
\end{definition}

\begin{definition}\label{Def:2.2}
 (Rockfellar\cite{R1976}) Let $g:X\rightarrow \mathbb{R}$ be a lower semicontinuous convex function, the subdifferential $\partial g$ of $g$ is defined to be the following set-valued operator: if $u \in dom(g)$,
 \begin{eqnarray*}
\partial g(u)=\{u^*:\langle u^*,v-u\rangle+g(u)\leq g(v),\forall v\in X,
\end{eqnarray*}
and if $u \notin dom(g)$, set $\partial g(u)=\emptyset$.
\vskip1mm
If $g$ is G$\hat{a}$teaux differentiable at $u$, denote by $\nabla g(u)$ the derivative of $g$ at $u$. In this case $\partial g(u)=\nabla g(u).$
\end{definition}
\begin{definition}\label{Def:2.3}
 (Clarke\cite{C2013})  Let $g:X\rightarrow \mathbb{R}$ be locally Lipschitz, the Clarke's generalized directional derivative of $g$ at $x$ in the direction $v$, denoted by $g^\circ(x;v)$, is defined as follows:
$$g^\circ(x;v)=\limsup_{y\rightarrow x,t\downarrow 0}\frac{g(y+tv)-g(y)}{t},$$
where $y$ lives in $E$ and $t$ is a positive scalar.
\end{definition}

\begin{definition}\label{Def:2.4} 
 (Clarke\cite{C2013}) The generalized gradient of the function $g$ at $x$, denoted by $\partial_C g(x)$, is the unique nonempty weak * compact convex subset of $E^*$ whose support function is $g^\circ(x;v)$, that is,
$$\partial_C g(x)=\{\xi:g^\circ(x;v)\geq\langle \xi,v\rangle,\forall v\in E\}.$$
\end{definition}
If $g$ is a convex function, then $\partial_C g(u)=\partial g(u).$
\begin{lemma}
Let $g, f$ be proper convex functions on $X$, and $\Phi=g-f$  attains its  minimum at $\tilde{x}$,  then   $\partial g(\tilde{x})\cap \partial f(\tilde{x})\neq\emptyset$.
\end{lemma}
\begin{proof} 
From Clarke\cite{C2013}, $\partial_C(g-f)(x)\subset \partial_C g(x)+ \partial_C (-f(x))=  \partial_C g(x)- \partial_C f(x)$. In addition, since $ f$ is proper convex function, it is locally Lipschitzian, and then $-f$ is also locally Lipschitzian, by the optimal condition, we have $0\in \partial_C(g-f)(\tilde{x})\subset  \partial_C g(\tilde{x})- \partial_C f(\tilde{x})$, which means that $\partial g(\tilde{x})\cap \partial f(\tilde{x})\neq\emptyset$.
\end{proof}

In the subsequent work, the components in (\ref{Eq:inclusion}) will be regularized respectively, to characterize the approximation solution of the whole problem, two metric functions $g_{\lambda, D}(z, x)$ and $f_{\mu,D}(z,x)$ for convex functions $g$ and $f$ are introduced as follows.
  
\begin{definition}\label{Def:2.3}
Let $\lambda>0,\mu>0$, $D\in S_{++}(\mathbb{R}^n)$ and $x\in \mathbb{R}^n$, the metric functions associated to $f$ and $g$ with parameters $\lambda$, $\mu$ and $D_i(i=1,2)$ are given by 
$$g_{\lambda,D_1}(z,x)=g(z)-g(x)+\frac{\|z-x\|^2_{D_1}}{2\lambda}, f_{\mu,D_2}(z,x)=f(z)-f(x)+\frac{\|z-x\|^2_{D_2}}{2\mu}.$$

For some $\epsilon>0$, if  we have 
\begin{eqnarray*}
|g_{\lambda,D_1}(\bar{z},x)-g_{\lambda,D_1}(z,x)|\leq \epsilon,\hspace{0.2cm}\text{and}\hspace{0.2cm}|f_{\mu,D_2}(\bar{z},x)-f_{\mu,D_2}(z,x)|\leq \epsilon,
\end{eqnarray*}
then the  point $\bar{z}$  is called an $\epsilon-$approximation of $z$.
\end{definition}
 The given constant $\epsilon$  controls the distance of the approximation from $z$ to $\bar{z}$  to $\hat{z}$. The metric function associated to $g$ with parameter $\lambda$ and $D_1)$ has the following property.

\begin{lemma}\label{lem:2.4}
For any $z_1,z_2\in \mathbb{R}^n$, we have
\begin{eqnarray*}
g_{\lambda,D_1}(z_2,x)-g_{\lambda,D_1}(z_1,x)\geq \nabla_{z_1} g_{\lambda,D_1}(z_1,x)^{T}(z_2-z_1)+\frac{1}{2\lambda m}\|z_2-z_1\|^2.
\end{eqnarray*}
\end{lemma}
\begin{proof}
\begin{eqnarray*}
&&g_{\lambda,D_1}(z_2,x)-g_{\lambda,D_1}(z_1,x)\\
&= &g(z_2)-g(x)+\frac{\|z_2-x\|^2_{D_1}}{2\lambda}-g(z_1)+g(x)-\frac{\|z_1-x\|^2_{D_1}}{2\lambda}\\
&=& g(z_2)-g(z_1)+\frac{1}{2\lambda}\|z_2-z_1\|_{D_1}^2+\frac{1}{\lambda}(z_1-x)^{T}D_1(z_2-z_1)\\
&\geq& \langle \partial g(z_1),z_2-z_1 \rangle+ \frac{1}{2\lambda}\|z_2-z_1\|_{D_1}^2+\frac{1}{\lambda}(z_1-x)^{T}D_1(z_2-z_1)\\
&=&\nabla_{z_1} g_{\lambda,D_1}(z_1,x)^{T}(z_2-z_1)+ \frac{1}{2\lambda}\|z_2-z_1\|_{D_1}^2\\
&\geq& \nabla_{z_1} g_{\lambda,D_1}(z_1,x)^{T}(z_2-z_1)+ \frac{1}{2m\lambda}\|z_2-z_1\|^2.
\end{eqnarray*}
\end{proof}

\begin{lemma}\label{Lem:2.5}(Descent Lemma, See Bertsekas\cite{B1999}) If $f$ is differential and $\|\nabla f(x)-\nabla f(y)\|\leq \eta\|x-y\|$, then the following holds:
 \begin{eqnarray*}
f(y)\leq f(x)+\langle\nabla f(x),y-x\rangle+\frac{\eta}{2}\|y-x\|^2.
\end{eqnarray*}
\end{lemma}

\section{Main result}\label{Sec:Alg}
\subsection{Some properties on the difference of Moreua Envelopes with different parameters}
In this section,we are concerned with the following  difference $\Phi(x)=g(x)-f(x)$ of convex problems, that is, finding $x\in \mathbb{R}^n$ such that
\begin{eqnarray}\label{Eq:inclusion1}
\inf_{x\in \mathbb{R}^n}\Phi(x).
\end{eqnarray} 
{\bf ASSUMPTION }
\vskip1mm
(A1) The functions  $g$  and $f$ are proper convex lower semicontinuous  on $\mathbb{R}^n$, and $dom g \subset dom f$.
\vskip1mm
(A2) The original function $\Phi$ is bounded below and  satisfies 
$\Phi(x)\geq \phi(\|x\|)+\beta$, where $\phi:[0,+\infty)\rightarrow[0,+\infty)$ is a nondecreasing continuous function with $\phi(0)=0$, $\lim_{t\rightarrow\infty}\phi(t)=+\infty$, $\beta$  is a real number.
\vskip1mm
(A3)  $D_i(i=1,2)\in M_m\subset \mathcal{S_{++}}, m\geq 1$.
\vskip1mm
To approximate the solution of (\ref{Eq:inclusion1}), we consider the following  difference of  Moreau envelopes of $f$ and $g$  induced by $\lambda,\mu$ and $D_i(i=1,2)$
as 
\begin{eqnarray}\label{3.2}
\inf_{x\in \mathbb{R}^n}\{\Phi_{\lambda,\mu}(x)=g_{\lambda,D_1}(x)-f_{\mu,D_2}(x)\},
\end{eqnarray}
where  $$g_{\lambda,D_1}(x)=\inf_{w\in \mathbb{R}^n}\{g(w)+\frac{1}{2\lambda}\|w-x\|^2_{D_1}\},\hspace{0.2cm}f_{\mu,D_2}(x)=\inf_{w\in \mathbb{R}^n}\{f(w)+\frac{1}{2\mu}\|w-x\|^2_{D_2}\}.$$
It follows from Glowinski et. al \cite{GOY2016}, $g_{\lambda,D_1}(x)\leq g(x)$ for all $x\in\mathbb{R}^n$, and $argmin g_{\lambda,D_1}(x)=argmin g(x)$, which is in general called the Moreau proximal operator  $prox_{\lambda g}^{D_1}(x)$.In fact, from the definition of $g_{\lambda, D_1}(z,x)$, we have 
$$argmin_{z\in \mathbb{R}^n} g_{\lambda, D_1}(z,x)=argmin g_{\lambda,D_1}(x)=argmin g(x).$$ Likely, 
$$argmin_{z\in \mathbb{R}^n} f_{\mu, D_2}(z,x)=prox_{\mu f}^{D_2}(x)= argmin_{x\in \mathbb{R}^n}f_{\mu,D_2}(x).$$

 {\bf Remark 1} For the Moreau proximal operators of $f$ and $g$, we have  $$g_{\lambda,D_1}(prox_{\lambda g}^D(x))=\inf_x g_{\lambda,D_1}(x)=g(prox_{\lambda g}^{D_1}(x))=\inf_x g(x),$$
and
 $$f_{\mu,D_2}(prox_{\mu f}^D(x))=\inf_x f_{\mu,D_2}(x)=f(prox_{\mu f}^{D_2}(x))=\inf_x f(x).$$

\begin{lemma}\label{lem:3.1}
 The Moreau proximal operators $prox_{\lambda g}^{D_1}(\cdot)$ , $prox_{\mu f}^{D_2}(\cdot)$ are Lipschitzian and single-valued. 
\end{lemma}
\begin{proof}
Indeed, it follows from the optimal condition that
\begin{eqnarray*}
0\in\partial g(prox_{\lambda g}^{D_1}(x))+\frac{D_1(prox_{\lambda g}^{D_1}(x)-x)}{\lambda},
\end{eqnarray*}
which implies that $ prox_{\lambda g}^{D_1}(x)=(D_1+\lambda \partial g)^{-1}D_1x$,
and then  $\frac{D_1x-D_1(prox_{\lambda g}^{D_1}(x))}{\lambda}\in \partial g(prox_{\lambda g}^{D_1}(x)).$ Similarly, we have $\frac{D_1y-D_1(prox_{\lambda g}^{D_1}(y))}{\lambda}\in \partial g(prox_{\lambda g}^{D_1}(y)).$
\vskip1mm
Taking into account the maximality of $\partial g$, we have 
\begin{eqnarray*}
\langle D_1x-D_1(prox_{\lambda g}^{D_1}(x))-D_1y+D_1(prox_{\lambda g}^{D_1}(y)),prox_{\lambda g}^{D_1}(x)- prox_{\lambda g}^{D_1}(y)\rangle \geq 0,
\end{eqnarray*}
which amounts to 
\begin{eqnarray*}
&&\langle D_1x-D_1y,prox_{\lambda g}^{D_1}(x)- prox_{\lambda g}^{D_1}(y)\rangle\\
&\geq& \langle D_1(prox_{\lambda g}^{D_1}(x))-D_1(prox_{\lambda g}^{D_1}(y)),prox_{\lambda g}^{D_1}(x)-prox_{\lambda g}^{D_1}(y)\rangle\\
&\geq& \frac{1}{m}\|prox_{\lambda g}^{D_1}(x)- prox_{\lambda g}^{D_1}(y)\|^2.
\end{eqnarray*}
And by the assumption on $D_i$, we have  
\begin{eqnarray*}
\langle D_1x-D_1y,prox_{\lambda g}^{D_1}(x)- prox_{\lambda g}^{D_1}(y)\rangle\leq m\|x-y\|\cdot\|prox_{\lambda g}^{D_1}(x)- prox_{\lambda g}^{D_1}(y)\|,
\end{eqnarray*}
hence 
\begin{eqnarray*}
\|prox_{\lambda g}^{D_1}(x)- prox_{\lambda g}^{D_1}(y)\|\leq m^2\|x-y\|,
\end{eqnarray*}
which reveals that the Moreau proximal operator $prox_{\lambda g}^{D_1}(\cdot)$ is Lipschitzian, also for $prox_{\mu f}^{D_2}(\cdot)$.
\vskip1mm
In addition, for any $x\in \mathbb{R}^n$, there exists a unique $z\in\mathbb{R}^n$ such that $z=prox_{\lambda g}^{D_1}(x)$. Otherwise, assume that $z_1\neq z_2$ in $\mathbb{R}^n$ such that
\begin{eqnarray*}
z_1=prox_{\lambda g}^{D_1}(x)=(D_1+\lambda \partial g)^{-1}D_1x\Rightarrow D_1x-D_1z_1\in\lambda \partial g(z_1),\\
z_2=prox_{\lambda g}^{D_1}(x)=(D_1+\lambda \partial g)^{-1}D_1x\Rightarrow D_1x-D_1z_2\in\lambda \partial g(z_2).
\end{eqnarray*}
Notice that $\partial g$ is maximal monotone, so we have $\langle D_1z_2-D_1z_1,z_1-z_2\rangle\geq 0$, which is contradict to the positive definitity of $D$. So $prox_{\lambda g}^{D_1}(\cdot)$ and $prox_{\mu f}^{D_2}(\cdot)$ are single-valued.
\end{proof}

\begin{lemma}\label{lem:3.2}
 Let $\Phi_{\lambda,\mu}: \mathbb{R}^n\rightarrow R$ be defined as in (\ref{3.2}), then 
\vskip1mm
(i)  $\Phi_{\lambda,\mu}$ is continuously differentiable on $dom(\Phi)$ and
\begin{eqnarray*}
\nabla\Phi_{\lambda,\mu}(x) =\frac{D_2(prox_{\mu f}^{D_2}(x)-x)}{\mu}- \frac{D_1(prox_{\lambda g}^{D_1}(x)-x)}{\lambda}.
\end{eqnarray*}
\vskip1mm 

(ii)  If $\lambda\geq m^2\mu$, then $\inf \Phi_{\lambda,\mu}(x)\geq\inf \Phi(x)$ and $\Phi_{\lambda,\mu}$ is evaluated as
\begin{eqnarray*}
&&\Phi(prox_{\lambda g}^{D_1}(x))+(\frac{1}{2m\lambda}-\frac{m}{2\mu})\|prox_{\lambda g}^{D_1}(x)-x\|^2\leq\Phi_{\lambda,\mu}(x)\\
&&\leq\Phi(prox_{\mu f}^{D_2}(x))+(\frac{m}{2\lambda}-\frac{1}{2m\mu})\|prox_{\mu f}^{D_2}(x)-x\|^2.
\end{eqnarray*}
(iii) $\nabla\Phi_{\lambda,\mu}(\cdot)$ is $\eta$-Lipschitz continuous, where $\eta$ is defined later.
\end{lemma}
\begin{proof}
(i) Since
$$g_{\lambda,D_1}(x)=g(prox_{\lambda g}^{D_1}(x))+\frac{1}{2\lambda}\|prox_{\lambda g}^{D_1}(x)-x\|^2_{D_1},$$
and
$$f_{\mu,D_2}(x)=f(prox_{\mu f}^{D_2}(x))+\frac{1}{2\mu}\|prox_{\mu f}^{D_2}(x)-x\|^2_{D_2},$$
we have
\begin{eqnarray*}
\Phi_{\lambda,\mu}(x)= g(prox_{\lambda g}^{D_1}(x))+\frac{1}{2\lambda}\|prox_{\lambda g}^{D_1}(x)-x\|^2_{D_1}-\{f(prox_{\mu f}^{D_2}(x))+\frac{1}{2\mu}\|prox_{\mu f}^{D_2}(x)-x\|^2_{D_2}\},
\end{eqnarray*}
and then
\begin{eqnarray*}
\nabla\Phi_{\lambda,\mu}(x)=\frac{D_1(x-prox_{\lambda g}^{D_1}(x))}{\lambda}-\frac{D_2(x-prox_{\mu f}^{D_2}(x))}{\mu}.
\end{eqnarray*}

Moreover,
\begin{eqnarray*}
&&\langle \nabla\Phi_{\lambda,\mu}(x)-\nabla\Phi_{\lambda,\mu}(y), x-y\rangle\\
&=&\langle \frac{D_2(prox_{\mu f}^{D_2}(x)-prox_{\mu f}^{D_2}(y))}{\mu},x-y\rangle- \langle\frac{D_1(prox_{\lambda g}^{D_1}(x)-prox_{\lambda g}^{D_1}(y))}{\lambda},x-y\rangle\\
&& +\langle \frac{D_1x-D_1y}{\lambda}- \frac{D_2x-D_2y}{\mu}, x-y\rangle.
\end{eqnarray*}
Since $(\frac{1}{\lambda m}-\frac{m}{\mu})\|x-y\|^2\leq\langle \frac{D_1x-D_1y}{\lambda}- \frac{D_2x-D_2y}{\mu}, x-y\rangle\leq (\frac{m}{\lambda}-\frac{1}{m\mu})\|x-y\|^2$,  we have
\begin{eqnarray*}
&&\frac{\|prox_{\mu f}^{D_2}(x)-prox_{\mu f}^{D_2}(y)\|^2}{m\mu}-\frac{m^3}{\lambda}\|x-y\|^2+(\frac{1}{\lambda m}-\frac{m}{\mu})\|x-y\|^2\\
&\leq&\langle \nabla\Phi_{\lambda,\mu}(x)-\nabla\Phi_{\lambda,\mu}(y), x-y\rangle\\
&\leq&\frac{m^3 \|x-y\|^2}{\mu}- \frac{1}{m\lambda}\| prox_{\lambda g}^{D_1}(x)-prox_{\lambda g}^{D_1}(y)\|^2+(\frac{m}{\lambda}-\frac{1}{m\mu})\|x-y\|^2,
\end{eqnarray*}
which means that 
\begin{eqnarray*}
\frac{-m^4\mu+\mu-\lambda m^2}{\lambda\mu m}\|x-y\|^2
\leq\langle \nabla\Phi_{\lambda,\mu}(x)-\nabla\Phi_{\lambda,\mu}(y), x-y\rangle
\leq\frac{\lambda m^4+\mu m^2-\lambda}{\lambda\mu m}\|x-y\|^2.
\end{eqnarray*}

Denote $\eta_1=\max\{|\frac{-m^4\mu+\mu-\lambda m^2}{\lambda\mu m}|,|\frac{\lambda m^4+\mu m^2-\lambda}{\lambda\mu m}|\}$, then we have 
\begin{eqnarray*}
|\langle \nabla\Phi_{\lambda,\mu}(x)-\nabla\Phi_{\lambda,\mu}(y), x-y\rangle|\leq\eta_1\|x-y\|^2.
\end{eqnarray*}
\vskip1mm
(ii) From the optimal condition, $p_x\in argmin_{x\in\mathbb{R}^n}\Phi_{\lambda,\mu}(x)$ amounts to $0=\nabla g_{\lambda,D_1}(p_x)- \nabla f_{\mu,D_2}(p_x)$, namely, 
\begin{eqnarray*}
0 =\frac{D_2(prox_{\mu f}^{D_2}(x)-p_x)}{\mu}- \frac{D_1(prox_{\lambda g}^{D_1}
(x)-p_x)}{\lambda},
\end{eqnarray*}
  which yields that 
\begin{eqnarray*}
p_x=\mu\lambda (\mu D_1-\lambda D_2)^{-1}(\frac{D_2(prox_{\mu f}^{D_2}(x)}{\mu}-\frac{D_1(prox_{\lambda g}^{D_1}(x)}{\lambda}),
\end{eqnarray*}
hence
\begin{eqnarray}\label{Eq:3.3}
\nonumber\inf\Phi_{\lambda,\mu}(x)&=&g(prox_{\lambda g}^{D_1}(x))-f(prox_{\mu f}^{D_2}(x))+\frac{\lambda}{2}\|z\|^2_{D_1^{-1}}-\frac{\mu}{2}\|z\|^2_{D_2^{-1}}\\\nonumber
&\geq&g(prox_{\lambda g}^{D_1}(x))-f(prox_{\mu f}^{D_2}(x))+(\frac{\lambda}{2m}-\frac{m\mu}{2})\|z\|^2\\\nonumber
&\geq&g(prox_{\lambda g}^{D_1}(x))-f(prox_{\mu f}^{D_2}(x)=\inf g(x)-\inf f(x)\\
&\geq& \inf \Phi(x),
\end{eqnarray}
where $z=\frac{D_2(prox_{\mu f}^{D_2}(x)-p_x)}{\mu}= \frac{D_1(prox_{\lambda g}^{D_1}
(x)-p_x)}{\lambda}$.
\vskip1mm
In addition, we have 
\begin{eqnarray*}
\Phi_{\lambda,\mu}(x)\geq g_{\lambda}(x)-\{f(y)+\frac{1}{2\mu}\|y-x\|^2_{D_2}\},y\in  \mathbb{R}^n,
\end{eqnarray*}
which holds for $y=prox_{\lambda g}^{D_1}(x)$, that is,
\begin{eqnarray*}
\Phi_{\lambda,\mu}(x)&\geq& g(prox_{\lambda g}^{D_1}(x))+\frac{1}{2\lambda}\|prox_{\lambda g}^{D_1}(x)-x\|^2_{D_1}\\
&&-\{f(prox_{\lambda g}^{D_1}(x))+\frac{1}{2\mu}\|prox_{\lambda g}^{D_1}(x)-x\|^2_{D_2}\}\\
&\geq&\Phi(prox_{\lambda g}^{D_1}(x))+(\frac{1}{2m\lambda}-\frac{m}{2\mu})\|prox_{\lambda g}^D(x)-x\|^2.
\end{eqnarray*}
\vskip1mm
On the other hand,  
\begin{eqnarray*}
\Phi_{\lambda,\mu}(x)\leq\{g(y)+\frac{1}{2\lambda}\|y-x\|^2_{D_1} \}-f_{\mu}(x),y\in  \mathbb{R}^n,
\end{eqnarray*}
which holds for $y=prox_{\mu f}^{D_2}(x)$, namely,
\begin{eqnarray*}
\Phi_{\lambda,\mu}(x)&\leq&g(prox_{\mu f}^{D_2}(x))+\frac{1}{2\lambda}\|prox_{\mu f}^{D_2}(x)-x\|^2_{D_1}\\
&&-f(prox_{\mu f}^{D_2}(x))-\frac{1}{2\mu}\|prox_{\mu f}^{D_2}(x)-x\|^2_{D_2}\\
&=&\Phi(prox_{\mu f}^{D_2}(x))+(\frac{m}{2\lambda}-\frac{1}{2m\mu})\|prox_{\mu f}^{D_2}(x)-x\|^2,
\end{eqnarray*}
this proves the thesis.

(iii)
\begin{eqnarray*}
 &&\|\nabla\Phi_{\lambda,\mu}(x)-\nabla\Phi_{\lambda,\mu}(x)\|\\
&=&\|\frac{D_1(x-y)}{\lambda}-\frac{D_2(x-y)}{\mu}-\frac{D_1(prox_{\lambda g}^{D_1}(x)-prox_{\lambda g}^{D_1}(y))}{\lambda}
+\frac{D_2(prox_{\mu f}^{D_2}(x)-prox_{\mu f}^{D_2}(y))}{\mu}\|\\
&\leq& (\frac{1}{\lambda}+\frac{1}{\mu})m\|x-y\|+\frac{\|D_1\|}{\lambda}\|prox_{\lambda g}^{D_1}(x)-prox_{\lambda g}^{D_1}(y)\|+\frac{\|D_2\|}{\mu}\|prox_{\mu f}^{D_2}(x)-prox_{\mu f}^{D_2}(y)\|\\
&\leq& (\frac{1}{\lambda}+\frac{1}{\mu})(m+m^3)\cdot\|x-y\|.
\end{eqnarray*}
Denote $\eta=(\frac{1}{\lambda}+\frac{1}{\mu})(m+m^3)$, then $\nabla\Phi_{\lambda,\mu}(\cdot)$ is $\eta$-Lipschitzian.
\end{proof}

{\bf Remark 2} (1) If the functions  $g_{\lambda,D}(z,x)$ and $f_{\mu,D}(z,x)$ are defined as in Definition\ref{Def:2.3},
we have 
\begin{eqnarray}\label{3.3}
\Phi_{\lambda,\mu}(x)=\Phi(x)+g_{\lambda,D_1}(prox_{\lambda g}^{D_1}x,x)-f_{\mu,D_2}(prox_{\mu f}^{D_2}x,x),
\end{eqnarray}
and then 
\begin{eqnarray}\label{3.4}
\Phi_{\lambda,\mu}(prox_{\lambda g}^{D_1}x)=\Phi(prox_{\lambda g}^{D_1}x)-f_{\mu,D_2}(prox_{\mu f}^{D_2}x,prox_{\lambda g}^{D_1}x).
\end{eqnarray}
(2) From Lemma \ref{lem:2.4},  if $z$ is the $\epsilon-$approximation of $prox_{\lambda g}^{D_1}x$, then $$\|z-prox_{\lambda g}^{D_1}x\|^2\leq 2m\lambda\epsilon.$$
Likely, if $z$ is the $\epsilon-$approximation of $prox_{\mu f}^{D_2}x$, then we have 
\begin{eqnarray*}
\|z-prox_{\mu f}^{D_2}x\|^2\leq 2m\mu\epsilon.
\end{eqnarray*}

(3) If $\lambda=\mu$, then  we have 
\begin{eqnarray*}
&&\Phi(prox_{\lambda g}^{D_1}(x))+\frac{1}{2\lambda}(\frac{1}{m}-m)\|prox_{\lambda g}^{D_1}(x)-x\|^2\\&\leq&\Phi_{\lambda,\lambda}(x)\leq\Phi(prox_{\lambda f}^{D_2}(x))+\frac{1}{2\lambda}(m-\frac{1}{m})\|prox_{\lambda f}^{D_2}(x)-x\|^2,
\end{eqnarray*}
and $\nabla\Phi_{\lambda,\lambda}(x) =\frac{D_2(prox_{\lambda f}^{D_2}(x)-D_1(prox_{\lambda g}^{D_1}(x))}{\lambda}$  is $\frac{2m+2m^3}{\lambda}$-Lipschitzian.
\vskip1mm
(4)  If $D_1=D_2\equiv I$ the identity operator, then our results recover that in Moudafi\cite{M2023}.
\vskip1mm
\subsection{Algorithms and approximation analysis}

\begin{algorithm}
\label{Alg:3.1}$\left. {}\right. $
\vskip1mm
{\bf  Algorithm 1 Inertial-Gradient Method}
\vskip1mm
{\rm \textbf{Initialization:} Give some sequence $\{\gamma_n\}\subset(0,1)$, choose  $\lambda>0, \mu>0$ and $ \lambda\geq m^2\mu$, $\frac{2\eta}{5\eta+2\eta_1}\leq \gamma<1$,  select arbitrary starting points $x_0,x_1\in \mathbb{R}^n$.
  \vskip 1mm
\textbf{Iterative Step:} Given the iterates $x_n$  for each $n\geq 1$,
 compute
\begin{eqnarray}\label{Eq:3.4}
\begin{cases}
w_n=x_n+\theta_n(x_n-x_{n-1})\\
 y_n=\nabla g_{\lambda,D_1}(w_n) =\frac{D_1(w_n-prox_{\lambda g}^{D_1}(w_n))}{\lambda},\\
z_n=\nabla f_{\mu,D_2}(w_n)=\frac{D_2(w_n-prox_{\mu f}^{D_2}(w_n))}{\mu}\\
x_{n+1}=x_n-\frac{\gamma}{\eta}(y_n-z_n),
\end{cases}
\end{eqnarray}
\textbf{Stopping Criterion:} If  $y_n=z_n$ then stop. Otherwise, set $n:=n+1$ and return to Iterative Step.}
\end{algorithm}

\begin{theorem}\label{Th:3.5}
 Suppose that ASSUMPTION (A1)-(A3) hold. Starting from $x_0,x_1\in \mathbb{R}^n$, we consider the iterates $(x_n, y_n, z_n)_{n\in \mathbb{N}}$ generated by Algorithm 1. Then, for every $n$, we have
\begin{eqnarray}\label{Eq:3.5}
\nonumber\Phi_{\lambda,\mu}(x_{n+1})&\leq&\Phi_{\lambda,\mu}(x_{n})+(\frac{\eta\theta_n}{2}-\frac{\eta}{\gamma}+\frac{\eta}{2})\|x_{n+1}-x_n\|^2\\
&&+(\eta_1\theta_n^2+\eta\theta_n^2+\frac{\eta\theta_n}{2})\|x_n-x_{n-1}\|^2,
\end{eqnarray}
where $\eta$ and $\eta_1$ are defined as in Lemma \ref{lem:3.2} (i) and (iii).
\end{theorem}
\begin{proof}
It follows from the descent principle because of the continuous gradient of $\Phi_{\lambda,\mu}$ that
\begin{eqnarray*}
\Phi_{\lambda,\mu}(x_{n+1})&\leq&\Phi_{\lambda,\mu}(x_{n})+\langle x_{n+1}-x_n,\nabla \Phi_{\lambda,\mu}(x_{n})\rangle+\frac{\eta}{2}\|x_{n+1}-x_n\|^2. 
\end{eqnarray*}
From $x_{n+1}=x_n-\frac{\gamma}{\eta}(y_n-z_n)=x_n-\frac{\gamma}{\eta}\nabla \Phi_{\lambda,\mu}(w_{n})$, we have 
\begin{eqnarray*}\label{3.}
\nonumber\langle x_{n+1}-x_n,\nabla \Phi_{\lambda,\mu}(x_{n})\rangle&=&\langle x_{n+1}-x_n,\nabla \Phi_{\lambda,\mu}(w_{n})\rangle+\langle x_{n+1}-x_n,\nabla \Phi_{\lambda,\mu}(x_{n})-\nabla\Phi_{\lambda,\mu}(w_{n})\rangle\\\nonumber
&=&-\frac{\eta}{\gamma}\|x_{n+1}-x_n\|^2+\langle w_n-x_n,\nabla \Phi_{\lambda,\mu}(x_{n})-\nabla\Phi_{\lambda,\mu}(w_{n})\rangle\\\nonumber
&&+\langle x_{n+1}-w_n,\nabla \Phi_{\lambda,\mu}(x_{n})-\nabla\Phi_{\lambda,\mu}(w_{n})\rangle\\\nonumber
&\leq&-\frac{\eta}{\gamma}\|x_{n+1}-x_n\|^2+\eta_1\|w_n-x_n\|^2\\\nonumber
&&+(\|x_{n+1}-x_n\|+\theta_n\|x_n-x_{n-1}\|)\cdot\eta\|x_n-w_n\|\\
&=&-\frac{\eta}{\gamma}\|x_{n+1}-x_n\|^2+\eta_1\theta_n^2\|x_n-x_{n-1}\|^2\\\nonumber
&&+(\|x_{n+1}-x_n\|+\theta_n\|x_n-x_{n-1}\|)\cdot\eta \theta_n\|x_n-x_{n-1}\|\\
&\leq&(\frac{\eta\theta_n}{2}-\frac{\eta}{\gamma})\|x_{n+1}-x_n\|^2+(\eta_1\theta_n^2+\eta\theta_n^2+\frac{\eta\theta_n}{2})\|x_n-x_{n-1}\|^2,
\end{eqnarray*}
which completes (\ref{Eq:3.5}).
\end{proof}
\begin{theorem}\label{Th:3.6}
 Suppose that ASSUMPTION  (A1)-(A3) hold. Starting from $x_0,x_1\in \mathbb{R}^n$, we consider the iterates $(x_n, y_n, z_n)_{n\in \mathbb{N}}$ generated by Algorithm 1. If the parameter $\theta_n$ satisfies $0<\theta_n\leq \frac{2(1-\gamma)\eta\gamma_n}{\gamma(2\eta_1+3\eta)}$, then the sequence $\{x_n\}$ is asymptotically regular and 
\begin{eqnarray*}
\sum_{n=0}^\infty\|x_{n+1}-x_n\|^2<\infty.
\end{eqnarray*}

Moreover,  the sequence $\{x_n\}$ is bounded.
\end{theorem}
\begin{proof}
In view of the  setting $\frac{2\eta}{5\eta+2\eta_1}\leq\gamma<1$, we have $\theta_n\leq \gamma_n<1$, then it follows from (\ref{Eq:3.5}) that
\begin{eqnarray}\label{Eq:3.9}
\Phi_{\lambda,\mu}(x_{n+1})\leq\Phi_{\lambda,\mu}(x_{n})+(\eta-\frac{\eta}{\gamma})\|x_{n+1}-x_n\|^2+(\eta_1+\frac{3\eta}{2})\theta_n\|x_n-x_{n-1}\|^2,
\end{eqnarray}
therefore,
\begin{eqnarray*}
\|x_{n+1}-x_n\|^2&\leq&\frac{\gamma(\eta_1+\frac{3\eta}{2})}{\eta(1-\gamma)}\theta_n\|x_n-x_{n-1}\|^2+\frac{\gamma}{\eta(1-\gamma)}[\Phi_{\lambda,\mu}(x_{n})-\Phi_{\lambda,\mu}(x_{n+1})]\\
&\leq&\gamma_n\|x_n-x_{n-1}\|^2+\frac{\gamma}{\eta(1-\gamma)}[\Phi_{\lambda,\mu}(x_{n})-\Phi_{\lambda,\mu}(x_{n+1})].
\end{eqnarray*}
\vskip1mm
Since  $\gamma_n\in (0,1)$, without loss of generality, suppose that there exists a real number $r<1$ such that $\gamma_n\leq r$, we have 
\begin{eqnarray*}
\|x_{n+1}-x_n\|^2&\leq& r\|x_n-x_{n-1}\|^2+\frac{\gamma}{\eta(1-\gamma)}[\Phi_{\lambda,\mu}(x_{n})-\Phi_{\lambda,\mu}(x_{n+1})].
\end{eqnarray*}
Denote $\psi_{n+1}=\|x_{n+1}-x_n\|^2$, $\delta_n=\frac{\gamma}{\eta(1-\gamma)}[\Phi_{\lambda,\mu}(x_{n})-\Phi_{\lambda,\mu}(x_{n+1})]$, then we  have  
\begin{eqnarray*}
\psi_{n+1}&\leq& r\psi_{n}+\delta_n\\
&\leq& r(r\psi_{n-1}+\delta_{n-1})+\delta_n\\
&\cdots&\\
&\leq&r^n\psi_1+\sum_{n=0}^{n-1}r^i\delta_{n-i},
\end{eqnarray*}
therefore
\begin{eqnarray*}
\sum_{n=0}^\infty\psi_{n+1}&\leq&\frac{1}{1-r}(\psi_1+\sum_{n=0}^{\infty}\delta_{n}),
\end{eqnarray*}
Now we consider the following two cases.

\vskip1mm
{\bf Case I}. If $\{\Phi_{\lambda,\mu}(x_{n})\}$ is decreasing, from  Lemma \ref{lem:3.1} (i) and the lower bound of $\Phi$, $\Phi_{\lambda,\mu}(x_n)$ converges to some limit $\Phi^*$, which in turn guarentees that 
\begin{eqnarray*}
\sum_{n=0}^{\infty}\delta_{n}<\infty,
\end{eqnarray*}
Consequently,  we obtain $\sum_{n=0}^\infty\psi_{n+1}<\infty$, that is, $\sum_{n=0}^\infty\|x_{n+1}-x_n\|^2<\infty.$
\vskip1mm
 Moreover, it follows from ASSUMPTION ( A2) that $\inf \Phi(x)\geq \inf\phi(\|x\|)+\beta$,  combining with (\ref{Eq:3.3}), we have 
\begin{eqnarray*}
\inf\phi(\|x_n\|)+\beta\leq\inf \Phi(x_n)\leq \inf\Phi_{\lambda,\mu}(x_{n}).
\end{eqnarray*}
Taking the limit on $n\rightarrow\infty$, we have 
\begin{eqnarray*}
\lim_{n\rightarrow\infty}\inf\phi(\|x\|)+\beta\leq\lim_{n\rightarrow\infty}\inf \Phi(x)\leq \lim_{n\rightarrow\infty}\inf\Phi_{\lambda,\mu}(x_{n})=\Phi^*,
\end{eqnarray*}
 which entails that $\{x_n\}$ is bounded from the coercity of $\phi$.
\vskip1mm
{\bf Case II}.  If $\{\Phi_{\lambda,\mu}(x_{n})\}$ is nondecreasing,  then $\delta_{n}\leq 0$, and we  have  
\begin{eqnarray*}
\psi_{n+1}\leq r\psi_{n}\leq\cdots r^n\psi_1,
\end{eqnarray*}
 hence
\begin{eqnarray*}
\sum_{n=0}^\infty\psi_{n+1}&\leq&\frac{1}{1-r}\psi_1,
\end{eqnarray*}
namely, $\sum_{n=0}^\infty\|x_{n+1}-x_n\|^2<\infty.$
\vskip1mm
In addition, it turns out  from (\ref{Eq:3.9}) that
\begin{eqnarray*}
\Phi_{\lambda,\mu}(x_{n+1})-\Phi_{\lambda,\mu}(x_{n})\leq(\eta_1+\frac{3\eta}{2})\theta_n\|x_n-x_{n-1}\|^2,
\end{eqnarray*}
and  then
\begin{eqnarray*}
\sum_{n=1}^\infty[\Phi_{\lambda,\mu}(x_{n+1})-\Phi_{\lambda,\mu}(x_{n})]\leq(\eta_1+\frac{3\eta}{2})\sum_{n=1}^\infty\|x_n-x_{n-1}\|^2<\infty,
\end{eqnarray*}
which in turn entails that there exists the limit of $\Phi_{\lambda,\mu}(x_{n})$, and from the line of  Case I,  the sequence $\{x_n\}$ is bounded.
\end{proof}

\begin{theorem}\label{Th:3.3}
 Suppose that ASSUMPTION (A1)-(A3)  hold. Starting from $x_0,x_1\in \mathbb{R}^n$, we consider the iterates $(x_n, y_n, z_n)_{n\in \mathbb{N}}$ generated by Algorithm 1. If the sequence $\{x_n\}$ converges to a cluster point $x^*$, then 
$$\Phi(prox_{\mu f}^{D_2}(x^*))-\Phi(prox_{\lambda g}^{D_1}(x^*))\rightarrow 0,$$
as $\lambda D_1^{-1}$ is close enough to $\mu D_2^{-1}$.
\end{theorem}
\begin{proof}
 Since $\{x_n\}$ is bounded, so are $\{y_n\}$ and $\{z_n\}$, and for any cluster points $x^*$ and $y^*$ of the sequences $\{x_n\}$ and $\{y_n\}$, 
$$y^*=\frac{D_1(x^*-prox_{\lambda g}^{D_1}(x^*))}{\lambda}=\frac{D_2(x^*-prox_{\mu f}^{D_2}(x^*))}{\mu},$$
 we have 
$$\|prox_{\lambda g}^{D_1}x^*)-prox_{\mu f}^{D_2}(x^*)\|=\|(\lambda D_1^{-1}-\mu D_2^{-1})y^*\|,$$
 which yields  $\|prox_{\lambda g}^{D_1}(x^*)-prox_{\mu f}^{D_2}(x^*)\|\rightarrow 0$ as $\lambda D_1^{-1}$ is close enough to $\mu D_2^{-1}$.
\vskip1mm
In view of the fact  that 
\begin{eqnarray*}
\Phi(prox_{\mu f}^{D_2}(x^*))-\Phi(prox_{\lambda g}^{D_1}(x^*))&=g(prox_{\mu f}^{D_2}(x^*))-f(prox_{\mu f}^{D_2}(x^*))\\
&-g(prox_{\lambda g}^{D_1}(x^*))+f(prox_{\lambda g}^{D_1}(x^*)),
\end{eqnarray*}
and $g$ is lower semi-continuous, we have 
$$g(prox_{\mu f}^{D_2}(x^*))\leq\lim_{\lambda D_1^{-1}\rightarrow\mu D_2^{-1}}\inf g(prox_{\lambda g}^{D_1}(x^*)),$$
similarly,
$$f(prox_{\lambda g}^{D_1}(x^*))\leq\lim_{\mu D_2^{-1}\rightarrow\lambda D_1^{-1}}\inf f(prox_{\mu f}^{D_2}(x^*)),$$
which means $\Phi(prox_{\mu f}^{D_2}(x^*))-\Phi(prox_{\lambda g}^{D_1}(x^*))\leq 0$.
\vskip1mm
On the other hand, $prox_{\lambda g}^{D_1}(x^*)$  and $prox_{\mu f}^{D_2}(x^*)$ are the minimums of $g$ and $f$, respectively, that is,
$$g(prox_{\mu f}^{D_2}(x^*))\geq g(prox_{\lambda g}^{D_1}(x^*)),f(prox_{\lambda g}^{D_1}(x^*))\geq f(prox_{\mu f}^{D_2}(x^*)),$$
so we have 
$\Phi(prox_{\mu f}^{D_2}(x^*))-\Phi(prox_{\lambda g}^{D_1}(x^*))\geq 0$, which proves the thesis.
\end {proof}
{\bf Remark 3} Although $\|prox_{\lambda g}^{D_1}(x^*)-prox_{\mu f}^{D_2}(x^*)\|\rightarrow 0$ when  $\lambda D_1^{-1}$ is close enough to $\mu D_2^{-1}$, it is not true for any other $x$, that is,
$\|prox_{\lambda g}^{D_1}(x)-prox_{\mu f}^{D_2}(x)\|\rightarrow 0$ is invalid. 
\vskip 1mm
\begin{theorem}\label{Th:3.4}
 Suppose that ASSUMPTION (A1)-(A3) hold. Starting from $x_0,x_1\in \mathbb{R}^n$, we consider the iterates $(x_n, y_n, z_n)_{n\in \mathbb{N}}$ generated by Algorithm 1. If the sequence $\{x_n\}$ converges to a cluster point $x^*$, then  $prox_{\mu f}^{D_2}(x^*)$ is an $\epsilon-$approximation point of  $prox_{\lambda g}^{D_1}(x^*)$  and we have 
\begin{eqnarray*}
\Phi_{\lambda,\mu}(prox_{\lambda g}^{D_1}(x^*))-\Phi(prox_{\lambda g}^{D_1}(x^*))\rightarrow 0,
\end{eqnarray*}
when $\lambda D_1^{-1}$ is close enough to $\mu D_2^{-1}$, where $\epsilon$ is  a given arbitrary small positive constant.
\end{theorem}
\begin{proof}

In deed, it turns out that
\begin{eqnarray*}
&&g_{\lambda,D_1}(prox_{\lambda g}^{D_1}x^*,x^*)-g_{\lambda,D_1}(prox_{\mu f}^{D_2}x^*,x^*)\\
&=&g(prox_{\lambda g}^{D_1}x^*)-g(x^*)+\frac{\|prox_{\lambda g}^{D_1}x^*-x^*\|^2_{D_1}}{2\lambda}\\
&&-\{g(prox_{\mu f}^{D_2}x^*)-g(x^*)+\frac{\|prox_{\mu f}^{D_2}x^*-x^*\|^2_{D_1}}{2\lambda}\}\\
&=&g(prox_{\lambda g}^{D_1}x^*)-g(prox_{\mu f}^{D_2}x^*)+\frac{1}{2\lambda}[\langle \lambda y^*,\lambda D_1^{-1}y^*\rangle -\langle D_1 D_2^{-1}\mu y^*, D_2^{-1}(\mu y^*)\rangle].
\end{eqnarray*}
Noticing that $\|prox_{\lambda g}^{D_1}(x^*)-prox_{\mu f}^{D_2}(x^*)\|=\|(\lambda D_1^{-1}-\mu D_2^{-1})y^*\|\rightarrow 0$ as $\lambda D_1^{-1}$ is closed enough to $\mu D_2^{-1}$ , we have 
$$|g_{\lambda,D_1}(prox_{\lambda g}^{D_1}x^*,x^*)-g_{\lambda,D_1}(prox_{\mu f}^{D_2}x^*,x^*)|\rightarrow 0.$$
Be in line with the proof in Theorem\ref{Th:3.3}, we have 
$$|f_{\lambda,D_2}(prox_{\lambda g}^{D_1}x^*,x^*)-f_{\lambda,D_2}(prox_{\mu f}^{D_2}x^*,x^*)|\rightarrow 0,$$
which implies that $prox_{\mu f}^{D_2}(x^*)$ is an $\epsilon-$approximation point of  $prox_{\lambda g}^{D_1}(x^*)$. Moreover, it turns out from (\ref{3.4}) that
 \begin{eqnarray*}
&&\Phi_{\lambda,\mu}(prox_{\lambda g}^{D_1}(x^*))-\Phi(prox_{\lambda g}^{D_1}(x^*))=-f_{\mu,D_2}(prox_{\mu f}^{D_2}x^*,prox_{\lambda g}^{D_1}x^*)\\
&=&f(prox_{\lambda g}^{D_1}x^*)-f(prox_{\mu f}^{D_2}x^*)-\frac{1}{2\mu}\|prox_{\mu f}^{D_2}x^*-prox_{\lambda g}^{D_1}x^*\|^2_{D_2},
\end{eqnarray*}
which proves the results.
\end{proof}
 
When $\theta_n=0$, the inertial gradient method reduces to the gradient descent method, which is listed in the following Algorithm 2.

\begin{algorithm}
\label{Alg:3.2}$\left. {}\right. $

{\bf  Algorithm 2  \hspace{0.5cm}Gradient Descent Method}
\vskip2mm
{\rm \textbf{Initialization:}  Give $D_1,D_2\in M_m$, choose  $\lambda>0,\mu>0$ and $\lambda \geq m^2\mu$, $0<\gamma<2$,  select arbitrary starting point $x_0\in \mathbb{R}^n$.
  \vskip 1mm
\textbf{Iterative Step:} Given the iterates $x_n$  for each $n\geq 1$,
 compute
\begin{eqnarray}\label{Eq:3.1}
\begin{cases}
 y_n=\nabla g_{\lambda,D_1}(x_n) =\frac{D_1(x_n-prox_{\lambda g}^{D_1}(x_n))}{\lambda},\\
z_n=\nabla f_{\mu,D_2}(x_n)=\frac{D_2(x_n-prox_{\mu f}^{D_2}(x_n))}{\mu}\\
x_{n+1}=x_n-\frac{\gamma}{\eta}(y_n-z_n),
\end{cases}
\end{eqnarray}
\textbf{Stopping Criterion:} If  $y_n=z_n$ then stop. Otherwise, set $n:=n+1$ and return to Iterative Step.}
\end{algorithm}

\begin{theorem}\label{Th:3.2}
 Suppose that ASSUMPTION (A1)-(A3)  hold. Starting from $x_0\in \mathbb{R}^n$, we consider the iterates $(x_n, y_n, z_n)_{n\in \mathbb{N}}$ generated by  Algorithm 2. Then, for every $n$, $\Phi_{\lambda,\mu}(x_n)$ is nonincreasing, specifically,
we have
\begin{eqnarray}\label{Eq:3.1}
\Phi_{\lambda,\mu}(x_{n+1})\leq\Phi_{\lambda,\mu}(x_n)+\eta(\frac{1}{2}-\frac{1}{\gamma})\|z_n-y_n\|^2,
\end{eqnarray}
where $\eta$ is defined as in Lemma \ref{lem:3.1} (ii). Moreover, the sequence $\{x_n\}$ is bounded.
\vskip 1mm
\end{theorem}
\begin{proof}
It follows from the descent  Lemma \ref{Lem:2.5} because of the Lipshitz property of continuous gradient of $\Phi_{\lambda,\mu}$ that
\begin{eqnarray*}
\Phi_{\lambda,\mu}(x_{n+1})&\leq&\Phi_{\lambda,\mu}(x_{n})+\langle x_{n+1}-x_n,\nabla \Phi_{\lambda,\mu}(x_{n})\rangle+\frac{\eta}{2}\|x_{n+1}-x_n\|^2. 
\end{eqnarray*}
From $x_{n+1}=x_n-\frac{\gamma}{\eta}(y_n-z_n)=x_n-\frac{\gamma}{\eta}\nabla \Phi_{\lambda,\mu}(x_{n})$, we have 
\begin{eqnarray*}
\Phi_{\lambda,\mu}(x_{n+1})&\leq&\Phi_{\lambda,\mu}(x_{n})+\langle x_{n+1}-x_n, -\frac{\eta}{\gamma}(x_{n+1}-x_n)\rangle+\frac{\eta}{2}\|x_{n+1}-x_n\|^2\\
&=&\Phi_{\lambda,\mu}(x_{n})+(\frac{\eta}{2}-\frac{\eta}{\gamma})\|x_{n+1}-x_n\|^2,
\end{eqnarray*}
which  means that the sequence $\Phi_{\lambda,\mu}(x_n)$ is a monotonically decreasing sequence and is a sufficient decrease condition of the approximation function values.  From Lemma \ref{lem:3.2}  and the assumption on the existence for the lower bound of $\Phi$, $\Phi_{\lambda,\mu}(x_n)$ converges to some limit $\Phi^*$, which in turn  guarentees that 
\begin{eqnarray*}
(\frac{\eta}{\gamma}-\frac{\eta}{2})\sum_{n=0}^{\infty}\|x_{n+1}-x_n\|^2\leq\sum_{n=0}^{\infty}(\Phi_{\lambda,\mu}(x_{n})-\Phi_{\lambda,\mu}(x_{n+1}))<+\infty,
\end{eqnarray*}
and then $\sum_{n=0}^{\infty}\|y_n-z_n\|^2<\infty$. 
\vskip1mm
Indeed, summing the above inequality over $n= 0,\cdots,N-1 $ for some positive integer $N-1$,
\begin{eqnarray*}
\sum_{n=0}^{N-1}\|y_n-z_n\|^2\leq\frac{2\eta}{\gamma(2-\gamma)}(\Phi_{\lambda,\mu}(x_0)-\Phi^*),
\end{eqnarray*}
Let $\bar{n}= argmin_{n=0,1,\cdots,N-1}\|y_n-z_n\|^2$, then one can check that 
\begin{eqnarray*}
N\|y_{\bar{n}}-z_{\bar{n}}\|^2\leq\frac{2\eta}{\gamma(2-\gamma)}(\Phi_{\lambda,\mu}(x_0)-\Phi^*),
\end{eqnarray*}
that is  $\|y_{{n}}-z_{{n}}\|=o(\frac{1}{\sqrt{n}})$.
\vskip 1mm
It follows from ASSUMPTION (A2) that $\inf \Phi(x)\geq \inf\phi(\|x\|)+\beta$,  combining with (\ref{Eq:3.3}), we have 
\begin{eqnarray*}
\inf\phi(\|x\|)+\beta\leq\inf \Phi(x)\leq \inf\Phi_{\lambda,\mu}(x_{n}).
\end{eqnarray*}
Taking the limit on $n\rightarrow\infty$, we have 
\begin{eqnarray*}
\lim_{n\rightarrow\infty}\inf\phi(\|x_n\|)+\beta\leq\lim_{n\rightarrow\infty}\inf \Phi(x_n)\leq \lim_{n\rightarrow\infty}\inf\Phi_{\lambda,\mu}(x_{n})=\Phi^*,
\end{eqnarray*}
 which entails that $\{x_n\}$ is bounded from the coercivity of $\phi$.
\end{proof}

\begin{theorem}\label{Th:3.7}
 Suppose that ASSUMPTION  (A1)-(A3) hold. Starting from $x_0\in \mathbb{R}^n$, we consider the iterates $(x_n, y_n, z_n)_{n\in \mathbb{N}}$ generated by Algorithm 2. If the sequence $\{x_n\}$ converges to a cluster point $x^*$, then 
$$\Phi(prox_{\mu f}^{D_2}(x^*))-\Phi(prox_{\lambda g}^{D_1}(x^*))\rightarrow 0,$$
as $\lambda D_1^{-1}$ is close enough to $\mu D_2^{-1}$.
\end{theorem}

\begin{theorem}\label{Th:3.8}
 Suppose that ASSUMPTION  (A1)-(A3) hold. Starting from $x_0\in \mathbb{R}^n$, we consider the iterates $(x_n, y_n, z_n)_{n\in \mathbb{N}}$ generated by Algorithm 2. If the sequence $\{x_n\}$ converges to a cluster point $x^*$, then  $prox_{\mu f}^{D_2}(x^*)$ is an $\epsilon-$approximation point of  $prox_{\lambda g}^{D_1}(x^*)$  and we have 
\begin{eqnarray*}
\Phi_{\lambda,\mu}(prox_{\lambda g}^{D_1}(x^*))-\Phi(prox_{\lambda g}^{D_1}(x^*))\rightarrow 0,
\end{eqnarray*}
when $\lambda D_1^{-1}$ is close enough to $\mu D_2^{-1}$.
\end{theorem}

{\bf Remark 4} (1) According to the definition of the approximation stationary point in Moameni\cite{M2020}, Moudafi\cite{M2023}, Sun and Sun\cite{SS2021}, $D_1prox_{\lambda g}^{D_1}(x^*), D_2prox_{\mu f}^{D_2}(x^*)$ are the approximation stationary points of $\Phi$ as $\lambda D_1^{-1}$ is close enough to $\mu D_2^{-1}$. 
\vskip2mm
(2) During the inertial  gradient descent method (Algorithm 1) and the   gradient descent method (Algorithm 2), the subdifferentials of $g$ and $f$ are not evaluate at the same point, so the $\epsilon$-approximation is  introduced naturally.

\section{Numerical Example}
\begin{example} 
We consider the DC problem $\Phi(x)=|x|^3-|x|,x\in \mathbb{R}$, namely,
\begin{eqnarray*}
\Phi(x)=
\begin{cases}
x^3-x, \hspace{0.2cm}x\geq 0;\\
-x^3+x, \hspace{0.2cm}x< 0,
\end{cases}
\end{eqnarray*}
 we have  $\partial_C\Phi(0)=[-1,1] $. 
\vskip1mm
Let $\Phi'(x)=0$, we have $x=\pm\frac{1}{\sqrt{3}}$, hence 
$$\inf_{x\in \mathbb{R}}\Phi(x)=(\frac{1}{3})^{3/2}-(\frac{1}{3})^{1/2}=-\frac{2}{3}(\frac{1}{3})^{1/2}\approx -0.384900179.$$ 
\vskip1mm
Moreover,  we can check  that there exists a coercive function  $\frac{1}{2}|x|^3+\beta$ such that $\Phi(x)\geq \frac{1}{2}|x|^3+\beta$ (see Figure 1.) and then $\Phi(x)\rightarrow+\infty$ when $|x|\rightarrow+\infty$.
\begin{figure}[H]
\centering
\includegraphics[width=0.6\textwidth]{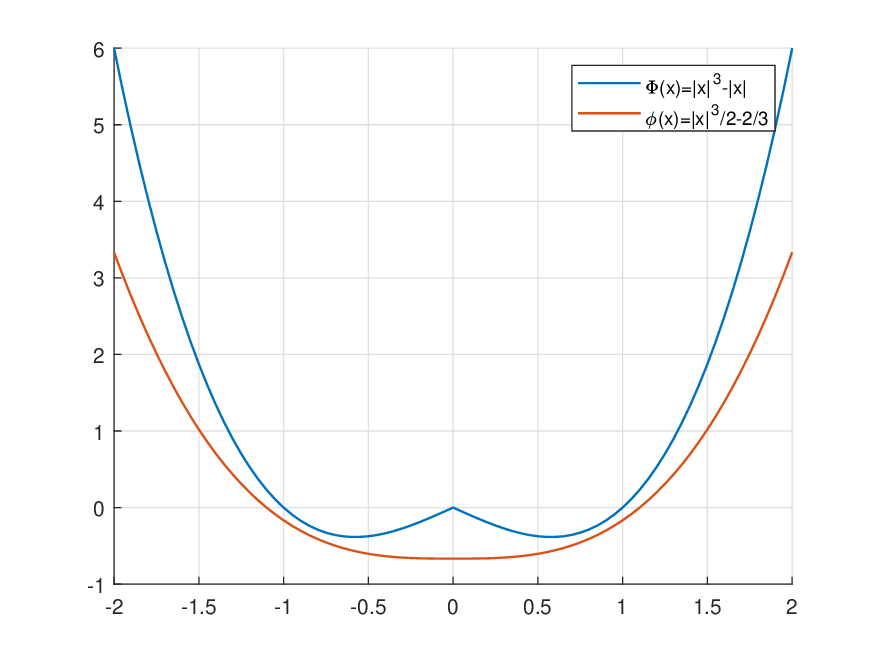}
\caption{\small{$\Phi(x)=|x|^3-|x|,\phi(x)=\frac{|x|^3}{2}- \frac{2}{3}.$ }}
\label{Fig:1}
\end{figure}
Now regularizing the function $\Phi$ with the  matrices $D_1,D_2$ 
\begin{eqnarray*}
\Phi_{\lambda,\mu}(x)&=&\inf_{w\in\mathbb{R}}\{|w|^3+\frac{1}{2\lambda}|w-x|_{D_1}^2\}-\inf_{w\in\mathbb{R}}\{|w|+\frac{1}{2\mu}|w-x|_{D_2}^2\}\\
&=&g_{\lambda,D_1}(x)-f_{\mu, D_2}(x),
\end{eqnarray*}
Since $X=\mathbb{R}$, we take $D_1=d_1\geq 1$, $D_2=d_2\geq 1$ and hence 
\begin{eqnarray*} 
g_{\lambda,D_1}(x)=
\begin{cases}
(\frac{-d_1+\sqrt{d_1^2+12\lambda d_1 x}}{6\lambda})^3+\frac{d_1(\frac{-d_1+\sqrt{d^2_1+12\lambda d_1 x}}{6\lambda}-x)^2}{2\lambda},  x\geq 0\\
(\frac{-d_1+\sqrt{d^2_1-12\lambda d_1 x}}{6\lambda})^3+\frac{d_1(\frac{-d_1+\sqrt{d^2_1-12\lambda d_1 x}}{6\lambda}-x)^2}{2\lambda},  x< 0,
\end{cases}
\end{eqnarray*}
and 
\begin{eqnarray*}
f_{\mu,D_2}(x)=
\begin{cases}
x-\frac{\mu}{2d_2}, x\geq 0,\\
-x-\frac{\mu}{2d_2},x <0.
\end{cases}
\end{eqnarray*}
that is,
\begin{eqnarray*}
\Phi_{\lambda,\mu}(x)=
\begin{cases}
(\frac{-d_1+\sqrt{d_1^2+12\lambda d_1 x}}{6\lambda})^3+\frac{d_1(\frac{-d_1+\sqrt{d^2_1+12\lambda d_1 x}}{6\lambda}-x)^2}{2\lambda}-x+\frac{\mu}{2d_2},  x\geq 0\\
(\frac{-d_1+\sqrt{d^2_1-12\lambda d_1 x}}{6\lambda})^3+\frac{d_1(\frac{-d_1+\sqrt{d^2_1-12\lambda d_1 x}}{6\lambda}-x)^2}{2\lambda}+x+\frac{\mu}{2d_2},  x< 0.
\end{cases}
\end{eqnarray*}
And then we have
\begin{eqnarray*}
\nabla\Phi_{\lambda,\mu}(x)=
\begin{cases}
(\frac{\sqrt{d_1^2+12\lambda d_1 x}-d_1}{6\lambda})^2\frac{3d_1}{\sqrt{d_1^2+12\lambda d_1 x}}+\frac{d_1(\frac{\sqrt{d_1^2+12\lambda d_1 x}-d_1}{6\lambda}-x)(\frac{d_1}{\sqrt{d_1^2+12\lambda d_1x}}-1)}{\lambda}-1, x\geq 0,\\
1-(\frac{\sqrt{d_1^2-12\lambda d_1x}-d_1}{-6\lambda})^2\frac{3d_1}{\sqrt{d_1^2-12\lambda d_1 x}}+\frac{d_1(\frac{\sqrt{d_1^2-12\lambda d_1 x}-d_1}{-6\lambda}-x)(\frac{d_1}{\sqrt{d_1^2-12\lambda d_1 x}}-1)}{\lambda},x <0.
\end{cases}
\end{eqnarray*}

The numerical performance(the stop criterion is $\|x_{n+1}-x_n\|\leq 10^{-4}$)  for the cases of $\lambda\neq\mu$ and  $\lambda=\mu$  are considered and  listed in the Figures 2-5 ($d_1=d2=1$) and Tabels 1-2.
\begin{figure}[H]
\centering
\includegraphics[width=0.23\textwidth]{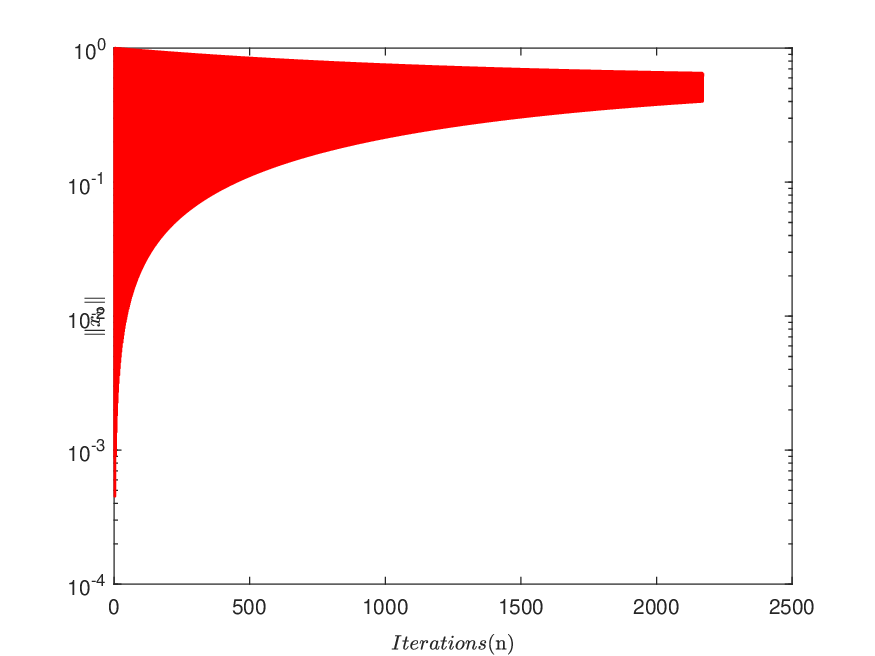}
\includegraphics[width=0.23\textwidth]{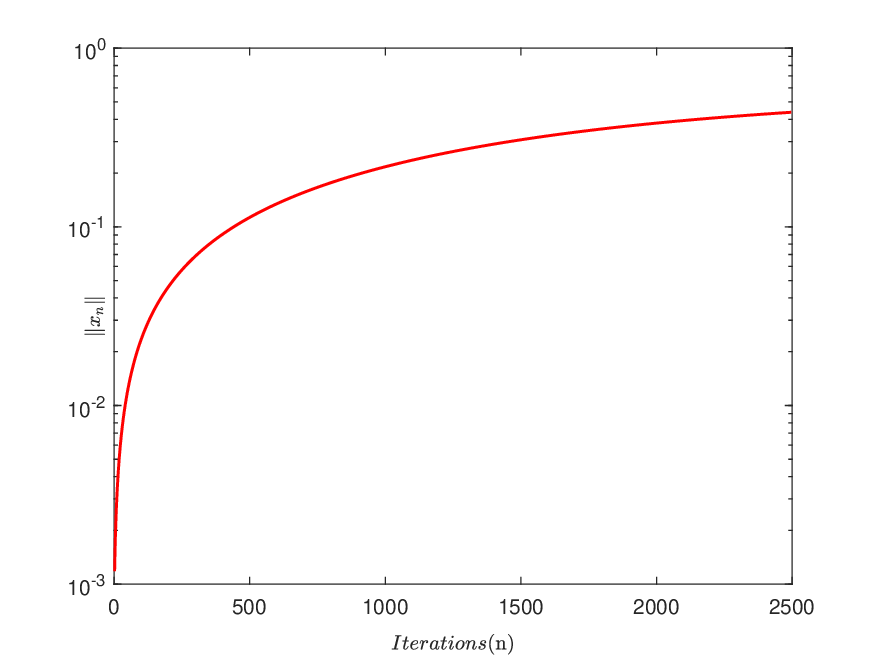}
\includegraphics[width=0.23\textwidth]{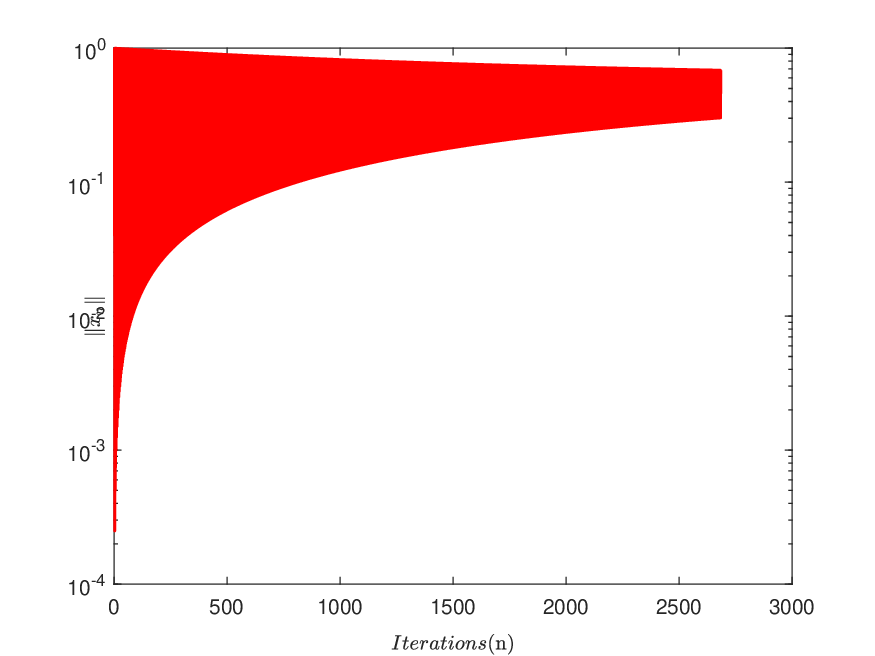}
\includegraphics[width=0.23\textwidth]{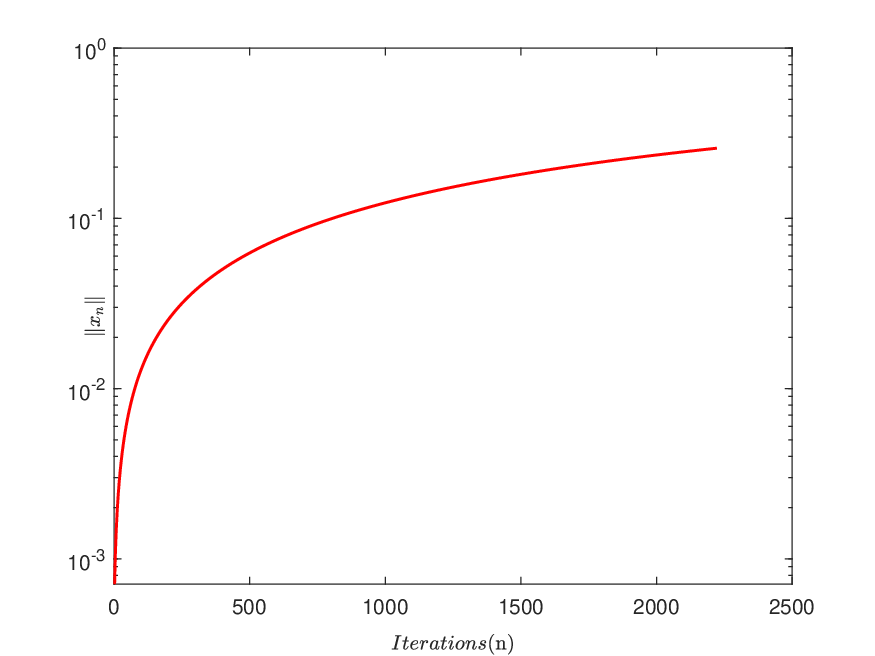}
\caption{\small{Left two:$\lambda=0.01,\mu=0.001$, Right two:$\lambda=\mu=0.001$  for Algo.1 and Algo. 2 }}
\label{Fig:2}
\end{figure}

\begin{figure}[H]
\centering
\includegraphics[width=0.23\textwidth]{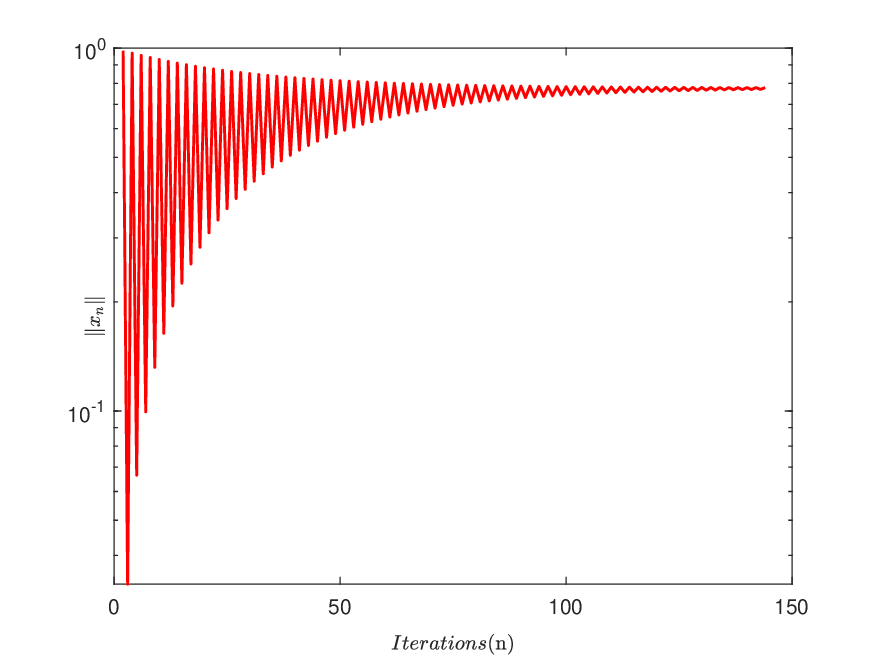}
\includegraphics[width=0.23\textwidth]{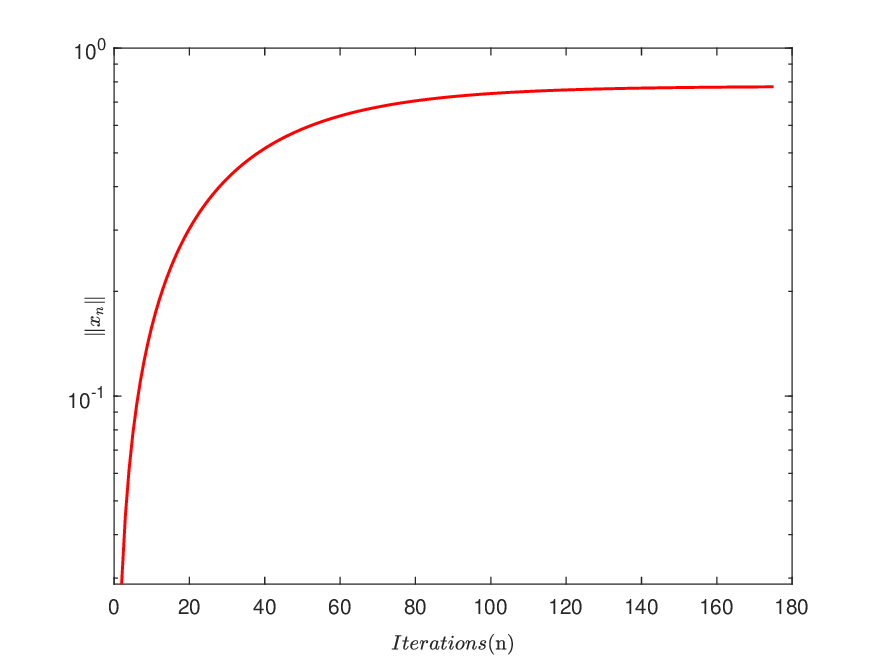}
\includegraphics[width=0.23\textwidth]{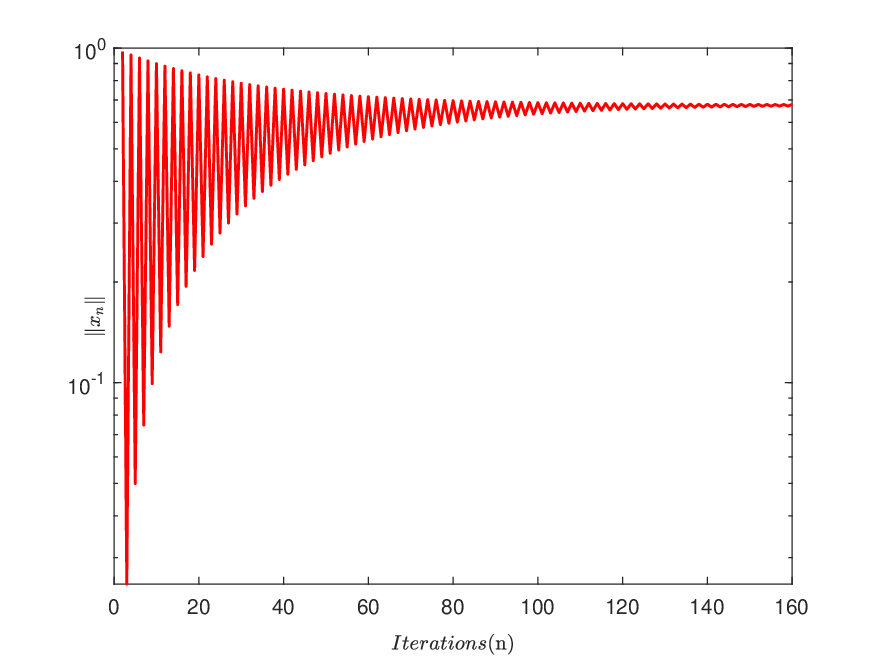}
\includegraphics[width=0.23\textwidth]{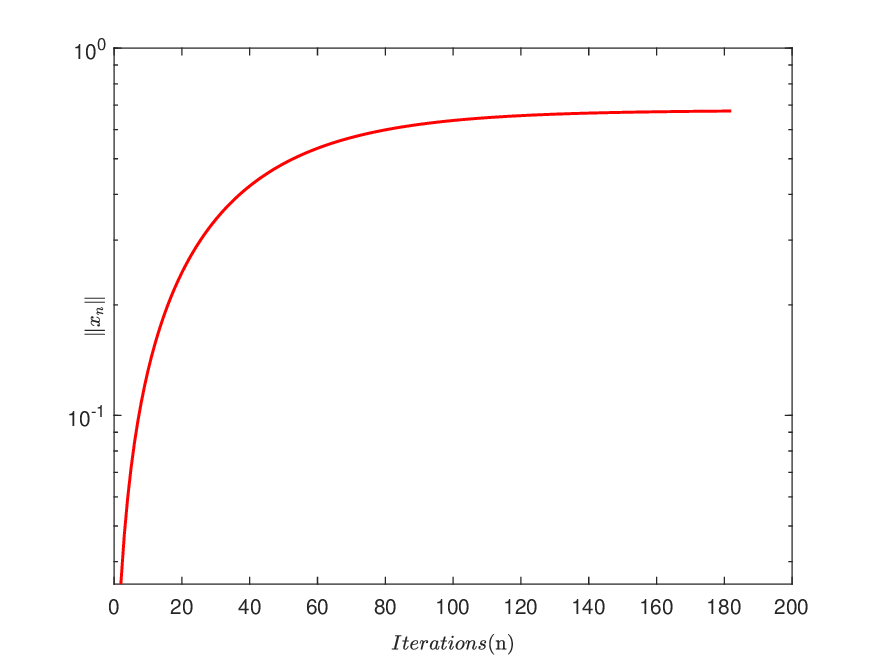}
  \caption{\small{Left two:$\lambda=0.2,\mu=0.1$, Right two:$\lambda=\mu=0.1$ for Algo.1 and Algo. 2 }}\label{Fig:3}
\end{figure}
\begin{figure}[H]
\centering
\includegraphics[width=0.23\textwidth]{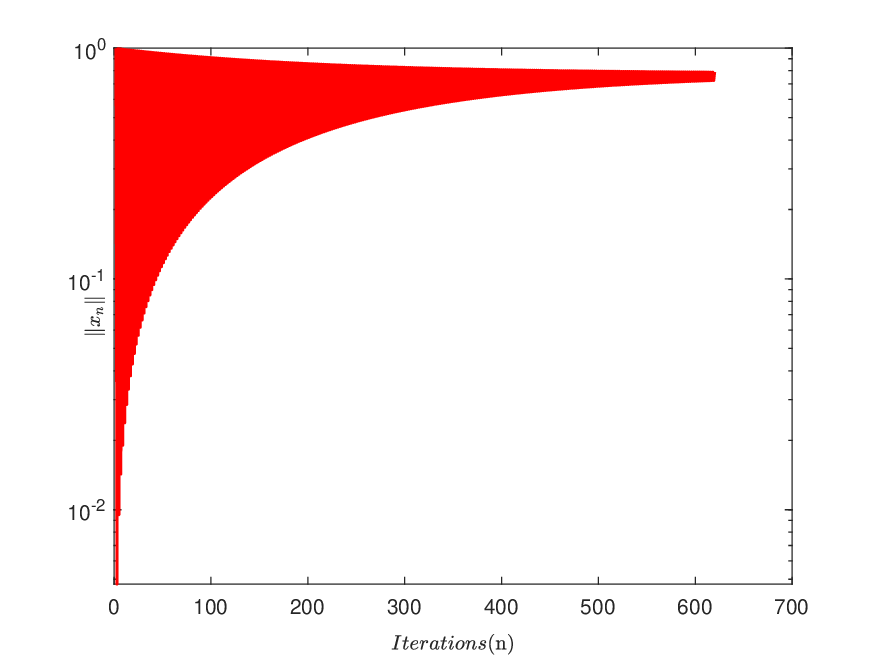}
\includegraphics[width=0.23\textwidth]{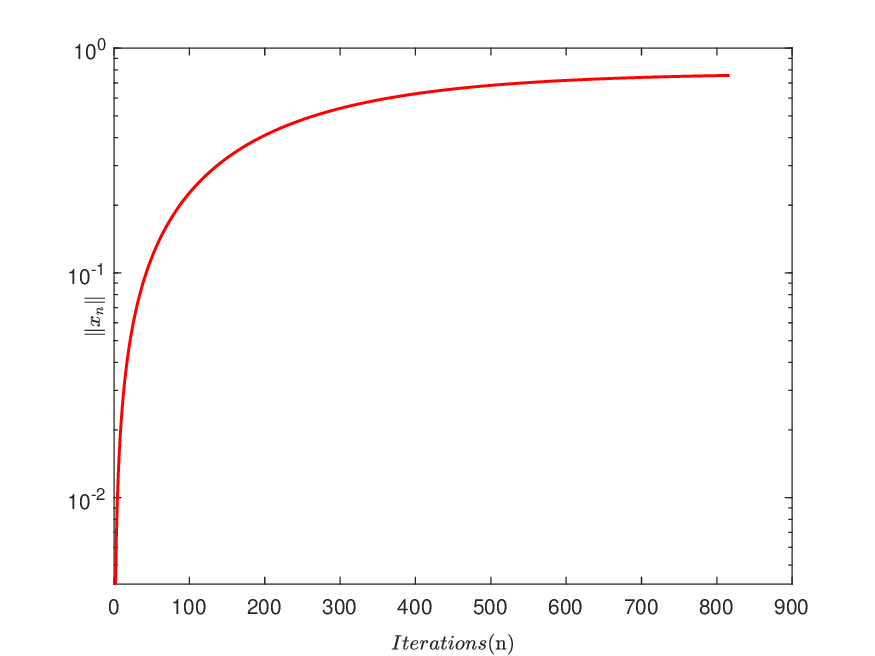}
\includegraphics[width=0.23\textwidth]{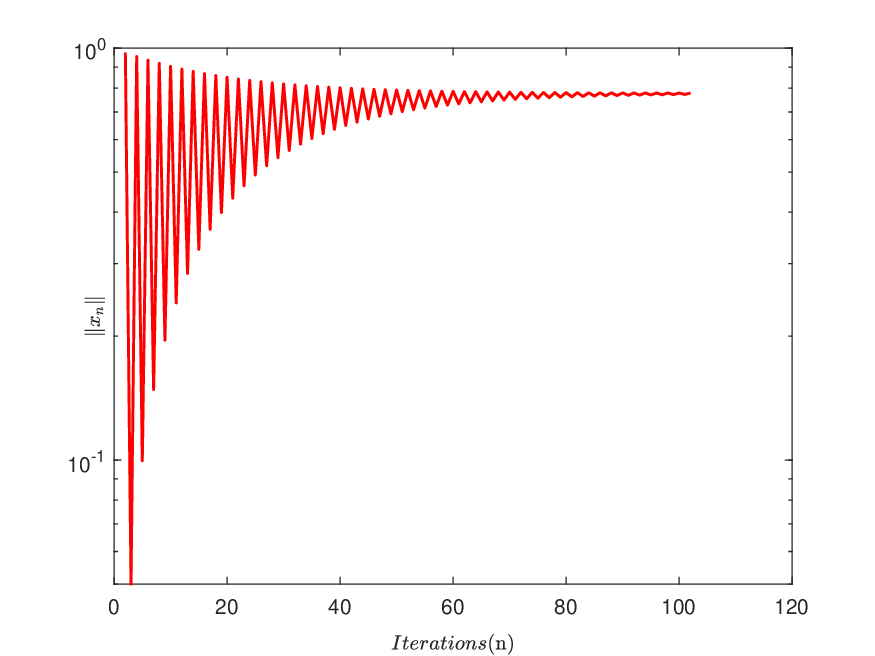}
\includegraphics[width=0.23\textwidth]{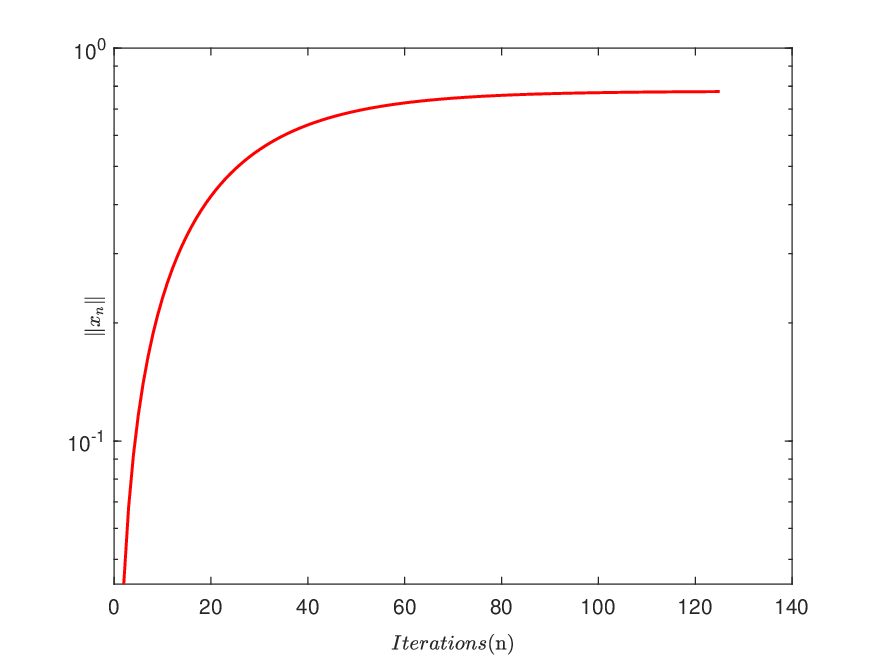}
  \caption{\small{Left two:$\lambda=0.2,\mu=0.01$, Right two:$\lambda=\mu=0.2$ for Algo.1 and Algo.2 }}
\label{Fig:4}
\end{figure}
\begin{figure}[H]
\centering
\includegraphics[width=0.23\textwidth]{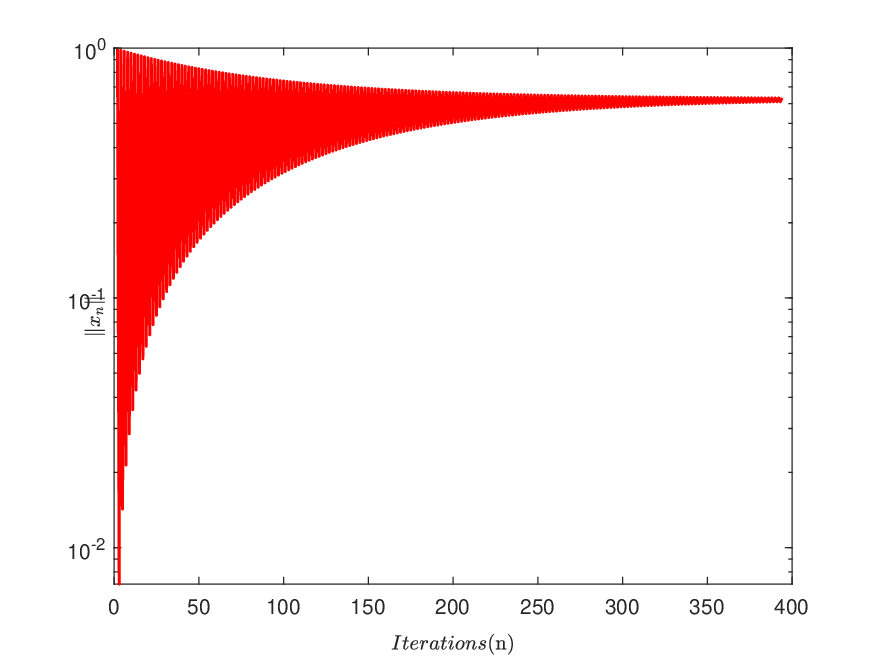}
\includegraphics[width=0.23\textwidth]{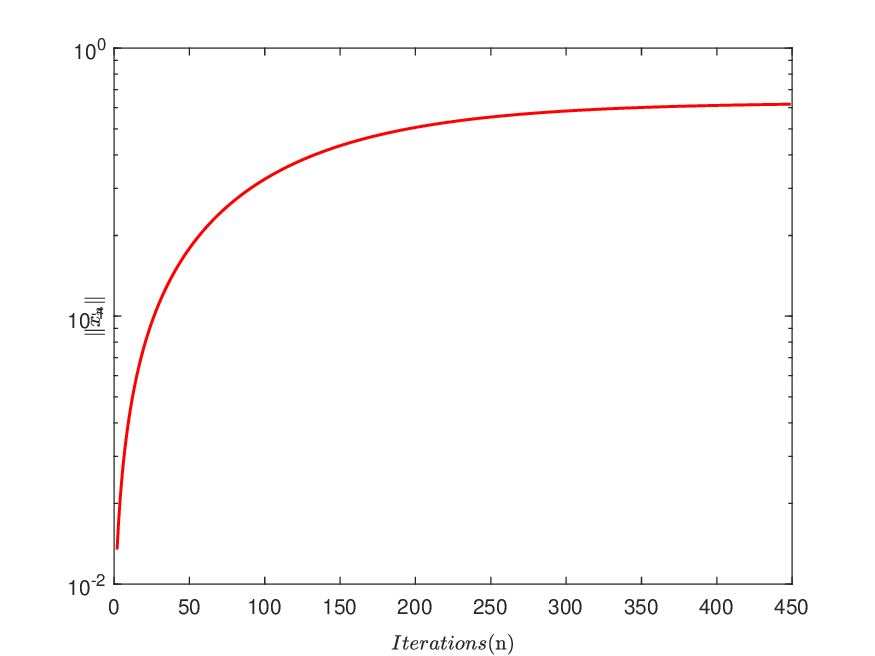}
\includegraphics[width=0.23\textwidth]{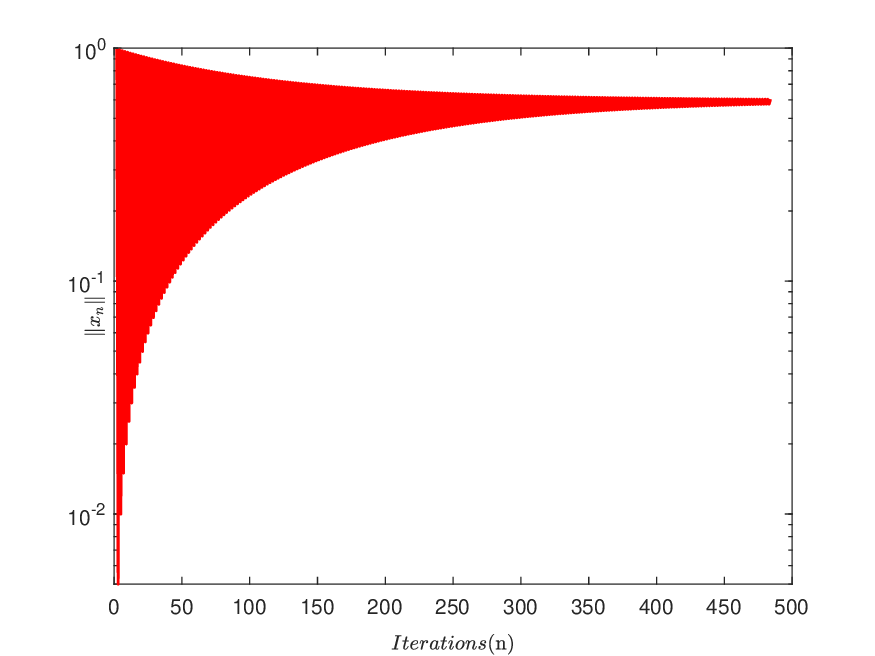}
\includegraphics[width=0.23\textwidth]{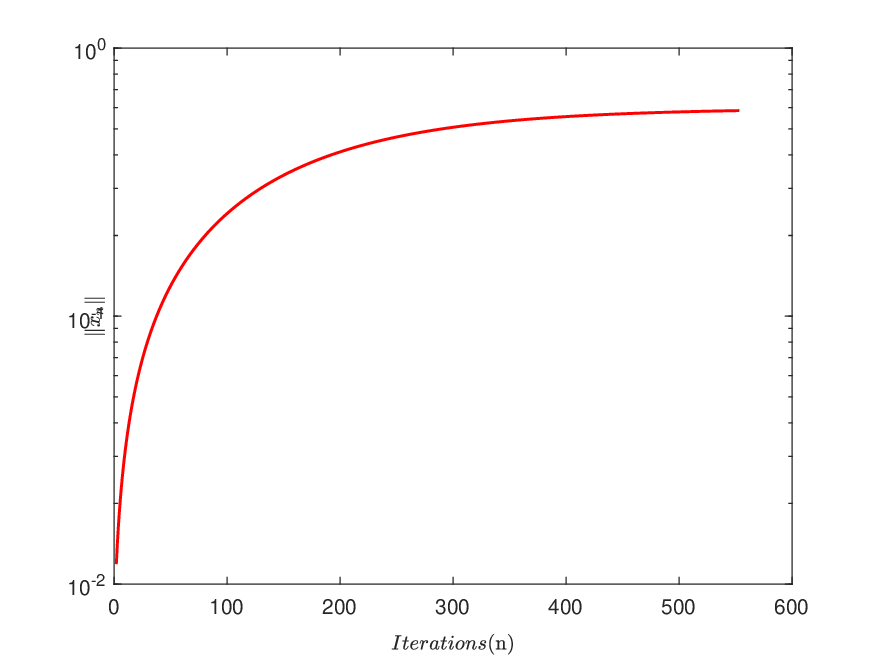}
  \caption{\small{Left two:$\mu=0.05,\lambda=0.02$,Right two$\lambda=\mu=0.02$ for Algo.1 and Algo.2 }}
\label{Fig:5}
\end{figure}

\begin{table}[H]
\small
  \centering
  \renewcommand{\arraystretch}{1.5}
   \caption{ Numerical results for Algo.1 and Algo.2}
\label{Tab:1}
\setlength{\tabcolsep}{0.7mm}
\begin{tabular}{| c| c c c | c c c | }
  \hline
&& Algo.2 ($\gamma=1.8,d_1=d_2=1$)& & & Algo.1($\gamma=0.9,d_1=d_2=1$)&\\
 \hline
$\lambda,\mu$ &$\Phi(x^*)$&CPU time  & Iter.& $\Phi(x^*)$&CPU time  & Iter. \\
 \hline

$0.15,0.1$&-0.3422&0.0209&90&-0.3436&0.0429&106\\
$0.05,0.02$&-0.3800&0.1229&248&-0.3814&0.1837&282\\
$0.03,0.02$&-0.3830&0.1326&278&-0.3839&0.2015&310\\

$0.03,0.01$&-0.3828&0.2553&394&-0.3841&0.3951&450\\
$0.02,0.01$&-0.3838&0.2838&420&-0.3847&0.4585&483\\
$0.005$&-0.3843&1.2439&832&-0.3841&0.6419&579\\
$0.01$&-0.3844&0.4615&510&-0.3849&0.2286&339\\
$0.01,0.005$&-0.3842&0.8677&686 &-0.3847&1.225&801\\
  \hline
\end{tabular}
\end{table}
\begin{table}[H]
\small
  \centering
  \renewcommand{\arraystretch}{1.5}
   \caption{ Numerical results for Algo.1 and Algo.2}
\label{Tab:1}
\setlength{\tabcolsep}{0.7mm}
\begin{tabular}{| c| cc c | c cc | }
  \hline
& & & Algo.1 ($\gamma=0.9$)& & & Algo.2($\gamma=1.8$)\\
 \hline
$\lambda=4\mu,d_1=1.5,d_2=2$ &$\Phi(x^*)$&CPU time  & Iter.& $\Phi(x^*)$&CPU time  & Iter. \\
 \hline

$\mu=0.05$&-0.3574&0.6188&531&-0.3487&0.4939&444\\
$\mu=0.03$&-0.3778&1.229&708&-0.3707&0.7250&604\\
$\mu=0.01$&-0.3843&4.4622&1393&-0.3808&3.3629&1174\\

$\mu=0.012$&-0.3848&2.8127&1245&-0.3804&3.8384&1052\\
$\mu=0.0125$&-0.3849&2.9881&1241&-0.3803&2.3332&1026\\
  \hline
\end{tabular}
\end{table}
From Table 1- 2, the operational results of the inertial gradient method (Algorithm 1) are closer to the exact solution of original DC problem than those of the gradient descent method (Algorithm 2), which is the reason that the idea of inertial has been increasingly applied to accelerate convergence in recent years.

\end{example}

\section{Conclusion}\label{Sec:Con}

In this paper, we  investigate the DC problem using Moreau regularization. Due to the convexity of both functions in the DC problem, global regularization cannot be easy to implement, so we regularizes each component separately with different parameters, smoothes the original problem, and defines a distance function to study and discuss the properties of approximate solutions of DC problem. At the same time, the approximation of the solution of the  original problem  by classical gradient algorithm and inertial gradient algorithm are discussed using the regularized function. Finally, numerical example demonstrate the approximation of the algorithms to the DC problem with different parameter selections.

\vskip 2mm

{\bf Competing Interests} The authors declare that they have no competing interests.
\vskip 2mm

\noindent

{\bf Authors' Contributions} All authors contributed equally to this work. All authors read and approved final manuscript.
\vskip 2mm

\noindent 

{\bf Acknowledgements} This article was funded by the National Natural Science Foundation of China (12071316) and Natural Science Foundation of Chongqing (cstc2021jcyj-msxmX0177).\\

\end{document}